\documentclass[journal]{IEEEtran}

\usepackage[utf8]{inputenc}
\usepackage{booktabs}
\usepackage{color}
\usepackage{array}
\usepackage{verbatim}
\usepackage{float}
\usepackage{amsmath}
\usepackage{amsthm}
\usepackage{amssymb}
\usepackage{graphicx}
\usepackage{longtable}
\usepackage[unicode=true,
bookmarks=false,
breaklinks=false,pdfborder={0 0 1},colorlinks=false]
{hyperref}
\hypersetup{
	colorlinks,bookmarksopen,bookmarksnumbered,citecolor=blue,urlcolor=blue}
\usepackage{cite}
\makeatletter
\let\oldforeign@language\foreign@language
\DeclareRobustCommand{\foreign@language}[1]{%
	\lowercase{\oldforeign@language{#1}}}

\let\oldforeign@language\foreign@language
\DeclareRobustCommand{\foreign@language}[1]{%
	\lowercase{\oldforeign@language{#1}}}

\ifCLASSINFOpdf
\else
\fi

\newtheorem{defn}{Definition}

\newtheorem{lem}{Lemma}

\newtheorem{thm}{Theorem}
\newtheorem{rem}{Remark}

\newtheorem{assum}{Assumption}

\begin{document}
	\bstctlcite{IEEEexample:BSTcontrol}

	\title{Distributed Adaptive Consensus Control of High Order Unknown Nonlinear Networked Systems with Guaranteed Performance}
	
	\author{Hashim A. Hashim\thanks{H. A. Hashim is with the Department of
			Electrical and Computer Engineering, Western University,
			London, ON, Canada, N6A-5B9, (e-mail: hmoham33@uwo.ca)}}
	
	
	\markboth{}{Hashim \MakeLowercase{\textit{et al.}}: Distributed Adaptive Consensus Control of High Order Unknown Nonlinear Networked Systems with Guaranteed Performance}
	
	\maketitle
	
	\begin{abstract}
Adaptive cooperative tracking control with prescribed performance
function (PPF) is proposed for high-order nonlinear multi-agent systems.
The tracking error originally within a known large set is confined
to a smaller predefined set using this approach. Using output error
transformation, the constrained system is relaxed and mapped to an
unconstrained one. The controller is conceived under the assumption
that the agents' nonlinear dynamics are unknown and the perceived
network is structured and strongly connected. Under the proposed controller,
all agents track the trajectory of the leader node with guaranteed
uniform ultimately bounded transformed error and bounded adaptive
estimate of unknown parameters and dynamics. In addition, the proposed
controllers with PPF are distributed such that each follower agent
requires information between its own state relative to connected neighbors.
Proposed controller is validated for robustness and smoothness using
highly nonlinear heterogeneous networked system with uncertain time-varying
parameters and external disturbances. 
	\end{abstract}
	

	\IEEEpeerreviewmaketitle{}

	\section{Introduction}
	
	\IEEEPARstart{B}{iological} behavioral analogies such as bees swarming, birds flocking,
	ants foraging, fish schooling have sprung up researchers' interest
	in distributed cooperative control (DCC) and its applications. DCC
	employs the divide-and-conquer mechanism to collaborate on tasks that
	are too complicated and time-consuming for individual agent. In addition,
	agents collaborate to enhance performance and productivity by exchanging
	information in manners present in social settings.\\
	\indent DCC is of paramount importance in many applications such
	as autonomous mobile robot vehicles control, energy and mineral explorations,
	surveillance, space explorations, just to mention a few. For the purpose
	of useful information exchange, agents are networked and referred
	to as nodes. A group of agents follows at least one real or virtual
	leader. The network of these nodes can be modeled as a directed or
	undirected graph. In undirected graphs, the relationship between two
	adjacent nodes is mutual and information flows in a bi-directional
	fashion. When the direction of information flow is not necessarily
	bi-directional, the term directed graph (or simply digraph) is conventionally
	used. In directed graphs, direction of the flow of information between
	a node and its neighbors is explicitly defined.\\
	\indent Design of DCC and multi-agent systems consensus in general
	was originally introduced in \cite{fax_information_2004,ren_consensus_2005}
	and has since then gone through many layers of advancement. Several
	studies have focused on addressing several research gaps such as node
	cooperative tracking problem \cite{olfati-saber_consensus_2007} and
	consensus of passive nonlinear systems \cite{chopra_passivity-based_2006}.
	Linear heterogeneous agents with multiple-input-multiple-output (MIMO)
	dynamics and parameter uncertainties have been considered and controlled
	in a distributed fashion \cite{zhao_fully_2014}. In addition, cooperative
	tracking control problems have been addressed for both single node
	with first-order dynamics \cite{das_distributed_2010,hashim2017adaptive,el2017neuro}
	and high-order dynamics \cite{zhang_adaptive_2012,yu2017observer}.
	Unknown nonlinear heterogeneous agents connected via a digraph and
	under a neuro-adaptive distributed control scheme have also been considered
	\cite{das_distributed_2010,zhang_adaptive_2012,hashim2017neuro,el2017neuro}.\\
	\indent It is worth mentioning that most of the aforementioned studies
	assume known input in the node dynamics. To address this shortcoming,
	many studies have developed cooperative tracking control for systems
	with unknown input function \cite{theodoridis_direct_2012,el-ferik_neuro-adaptive_2014}.
	Neuro-adaptive-fuzzy-based cooperative tracking control was used to
	approximate unknown nonlinear dynamics and input functions \cite{theodoridis_direct_2012}.
	The choice of output membership functions' centroid was solely based
	on offline trials. One common assumption to many of the previous studies
	is that they mostly consider unknown nonlinear dynamics and input
	function with parameter linearity \cite{lewis_neural_1998,el-ferik_neuro-adaptive_2014}
	with the objective of ensuring ultimate stability of the tracking
	error.\\
	\indent Owing to these shortcomings, prescribed performance framework
	will be beneficial for guaranteeing desired performance. Prescribed
	performance framework has been used in many control applications in
	recent time. Multi-agent distributed control in a way addresses many
	problems such as nonlinearities, unmodeled dynamics, uncertainties,
	and disturbances. It is often conceived that closed loop characteristics
	(transient and steady state error for instance) are analytically rigorous
	\cite{bechlioulis_robust_2008,hashim2017adaptive,hashim2017neuro}.
	With prescribed performance however, such characteristics have been
	mapped to a relatively smaller set with constrained convergence. In
	prescribed performance approach, error is transformed from a space
	that is constrained into unconstrained one. Some of many objectives
	of this method are that error of convergence must be lesser than predefined
	value, prescribed constant bounds the maximum overshoot, smooth system's
	output, boundedness and smoothness of control signal is guaranteed,
	and boundedness of error \cite{bechlioulis_robust_2008,hashim2017adaptive,hashim2017neuro}. 
	Details about prescribed performance is giving in subsequent section.\\
	\indent Cooperative control with prescribed performance for multi-agent
	systems is beneficial because it ensures that the consensus output
	error starts in a large predefined set with steady convergence into
	a narrow set that is predefined \cite{bechlioulis_robust_2008,hashim2017adaptive,hashim2017neuro}.
	The transient and steady-state consensus tracking errors are in conformity
	with a known time-varying performance. Adaptive cooperative control
	with prescribed performance is capable of improving the robustness
	and lowering the control effect. By carefully selecting the upper
	and lower bounds of the prescribed performance functions, the error
	converges inside a preset bounds.\\
	\indent The application of prescribed performance scheme with neural
	approximation including strict-feedback systems \cite{bechlioulis_robust_2008,mohamed_improved_2014}    
	and large-scale systems with time delays \cite{li2015prescribed}
	have been addressed in recent studies. Several studies considered
	adaptive prescribed performance based on neural network such as \cite{bechlioulis_robust_2008}.
	A vast majority of these studies were based on the assumptions that
	input matrix continuity is continuous. Significant improvements have
	been made by replacing neural network with trial and error hyper-parameter
	tuning for controller design with prescribed performance based on
	model reference adaptive control \cite{mohamed_improved_2014}.\\                                        
	\indent This paper proposes a robust adaptive distributed control
	with prescribed performance for a group of agents with high order
	nonlinear dynamics connected through a directed communication graph
	with known topology. The proposed control law is fully distributed
	in the sense that only local neighborhood information is accessible
	to each node. Also, the control law of each node is designed to preserve
	the network topology and communications from the leader do not propagate
	to all the nodes. In this work, $L$ and $B$ matrices characterize
	the general network framework. The imposed predefined characteristics
	confine the nodes synchronization error in conformity with the prescribed
	performance. The dynamics of the nodes are unknown, nonlinear, and
	with time-varying uncertainties. The controller for each node is conceived
	with predefined transient and steady-state requirements. The proposed
	ensures stable dynamics and control signal that is bounded and smooth.
	This paper considerably expands the scope of single-order nonlinear
	nodes considered in \cite{Hashim2017adaptive} to high order dynamics.\\
	\indent The rest of the paper is structured as follows: Section
	\ref{sec3} presents preliminaries of math notations and graph theory.
	Problem and local error synchronization formulations are presented
	in Section \ref{sec4}. Prescribed performance is introduced in Section
	\ref{sec5}. Section \ref{sec6} formulates the control law and stability
	proof of the connected graph. Simulations results are detailed in
	Section \ref{sec7}. Finally, conclusions are drawn with future work
	in Section \ref{sec8}.
	
	\section{Mathematical Identities And Basic Graph Theory \label{sec3}}
	
	\subsection{Math Notations and Identities}
	
	Throughout this paper, the set of real numbers is denoted as $\mathbb{R}$;
	$n$-dimensional vector space as $\mathbb{R}^{n}$; the space span
	by $n\times m$ matrix as $\mathbb{R}^{n\times m}$; identity matrix
	of order $m$ as $\mathbf{I}_{m}$; absolute value as $|\cdot|$.
	For $x\in\mathbb{R}^{n}$, the Euclidean norm is given as $\left\Vert x\right\Vert =\sqrt{x^{\top}x}$
	and matrix Frobenius norm is given as $\left\Vert \cdot\right\Vert _{F}$.
	For any $x_{i}\in\mathbb{R}^{n}$ we have $x_{i}=\left[x_{i}^{1},\ldots,x_{i}^{n}\right]^{\top}$
	for $i=1,2,\ldots,N$ and for $x^{j}\in\mathbb{R}^{N}$ we have $x^{j}=\left[x_{1}^{j},\ldots,x_{N}^{j}\right]$
	for $j=1,2,\ldots,n$. Trace of associated matrix is denoted as ${\rm Tr}\left\{ \cdot\right\} $,
	$\mathcal{N}$ is the set $\left\{ 1,...,N\right\} $, and ${\bf \underline{1}}_{N}$
	is a unity vector $[1,\ldots,1]^{\top}\in\mathbb{R}^{N}$. $A$ is
	said to be positive definite if $A>0$ for $A\in\mathbb{R}^{n\times n}$;
	$A\geq0$ indicates positive semi-definite; $\sigma\left(\cdot\right)$
	is the set of singular values of a matrix with maximum value $\bar{\sigma}\left(\cdot\right)$
	and minimum value $\underline{\sigma}\left(\cdot\right)$. 
	
	\subsection{Basic Graph Theory}
	
	$\mathcal{G}=\left(\mathcal{V},\mathcal{E}\right)$ denotes a graph
	with a nonempty finite set of nodes (or vertices) $\mathcal{V}=\left\{ \mathcal{V}_{1},\mathcal{V}_{2},\ldots,\mathcal{V}_{n}\right\} $,
	and set of edges (or arcs) $\mathcal{E}\subseteq\mathcal{V}\times\mathcal{V}$.
	$\left(\mathcal{V}_{i},\mathcal{V}_{j}\right)\in\mathcal{E}$ if an
	edge exists from node $i$ to node $j$. The structure or topology
	of weighted graph is represented using the adjacency matrix $A=\left[a_{i,j}\right]\in\mathbb{R}^{N\times N}$
	with weights $a_{i,j}>0$ if $\left(\mathcal{V}_{j},\mathcal{V}_{i}\right)\in\mathcal{E}$
	and $a_{i,j}=0$ otherwise. It is assumed in this work that the topology
	is fixed (that is, $A$ is time-invariant) and there is no self-connectivity
	(that is, $a_{i,i}=0$). The sum of $i$-th row of $A$, that is,
	$d_{i}=\sum_{j=1}^{N}a_{i,j}$ is denoted as the weighted in-degree
	of a node $i$. Letting the diagonal of in-degree matrix $D={\rm diag}\left(d_{1},\ldots,d_{N}\right)\in\mathbb{R}^{N\times N}$
	and the Laplacian matrix $L=D-A$. The set of node $i$th neighbors
	of is $N_{i}=\left\{ j\left|\left(\mathcal{V}_{j}\times\mathcal{V}_{i}\right)\in\mathcal{E}\right.\right\} $.
	Node $j$ is said to be node $i$th neighbor if node $i$ can get
	information from node $j$. A sequence of successive edges of the
	form $\left\{ \left(\mathcal{V}_{i},\mathcal{V}_{k}\right),\left(\mathcal{V}_{k},\mathcal{V}_{l}\right),\ldots,\left(\mathcal{V}_{m},\mathcal{V}_{j}\right)\right\} $
	is a direct path from node $i$ to $j$. A digraph is a spanning tree
	if a direct path exists from the root node to all other nodes within
	a graph. A strongly connected digraph if for any ordered pair of nodes
	$\left[\mathcal{V}_{i},\mathcal{V}_{j}\right]$ with $i\neq j$, there
	is direct path from node $i$ to node $j$ \cite{ren_distributed_2008,lewis_cooperative_2013}.
	
	\section{Problem Formulation \label{sec4}}
	
	Consider the following nonlinear dynamics for the $i$th node 
	\begin{equation}
	\begin{aligned}\dot{x}_{i}^{1} & =x_{i}^{2}\\
	\dot{x}_{i}^{2} & =x_{i}^{3}\\
	& \vdots\\
	\dot{x}_{i}^{M_{p}} & =f_{i}\left(x_{i}\right)+G_{i}u_{i}\\
	y_{i} & =x_{i}^{1}
	\end{aligned}
	\label{eq:equa1}
	\end{equation}
	where $x_{i}^{m_{p}}\in\mathbb{R}^{P}$ denotes the $m_{p}$th-state
	of $i$, $P\geq1$ is number of inputs equal to number of outputs,
	$m_{p}=1,2,\ldots,M_{P}$, $M_{P}\geq1$ is the system order, $x_{i}=\left[x_{i}^{1},\ldots,x_{i}^{M_{p}}\right]^{\top}\in\mathbb{R}^{PM_{p}}$,
	and $i=1,2,\ldots,N$ with $N$ is number of agents in the graph. $G_{i}\in\mathbb{R}^{P\times P}$
	is a known control input matrix, $u_{i}\in\mathbb{R}^{P}$ is the
	control signal node with output vector $y_{i}\in\mathbb{R}^{P}$ such
	that $p=1,2,\ldots,P$. $f_{i}:\mathbb{R}^{PM_{p}}\rightarrow\mathbb{R}^{P}$
	is the unknown but Lipschitz nonlinear vector of dynamics. For simplicity
	we denote $x:=x\left(t\right)$ and $u:=u\left(t\right)$. The corresponding
	global dynamics of \eqref{eq:equa1} can be described by 
	\begin{equation}
	\begin{aligned}\dot{x}^{1} & =x^{2}\\
	\dot{x}^{2} & =x^{3}\\
	& \vdots\\
	\dot{x}^{M_{p}} & =f\left(x\right)+Gu\\
	y & =x^{1}
	\end{aligned}
	\label{eq:equa2}
	\end{equation}
	where $x^{m_{p}}=\left[x_{1}^{m_{p}},\ldots,x_{N}^{m_{p}}\right]^{\top}\in\mathbb{R}^{PN}$,
	$u=\left[u_{1}^{\top},\ldots,u_{N}^{\top}\right]^{\top}\in\mathbb{R}^{PN}$,
	$G={\rm diag}\left\{ G_{i}\right\} \in\mathbb{R}^{PN\times PN}$,
	$y=\left[y_{1},\ldots,y_{N}\right]^{\top}\in\mathbb{R}^{PN}$, $i=1,\ldots,N$
	and $f\left(x\right)=\left[f_{1}\left(x_{1}\right),\ldots,f_{N}\left(x_{N}\right)\right]^{\top}\in\mathbb{R}^{PN}$.
	The state dynamics of the leader are given as in \eqref{eq:equa3}
	\begin{equation}
	\begin{aligned}\dot{x}_{0}^{1} & =x_{0}^{2}\\
	\dot{x}_{0}^{2} & =x_{0,3}\\
	& \vdots\\
	\dot{x}_{0}^{M_{p}} & =f_{0}\left(t,x_{0}\right)\\
	y_{0} & =x_{0}^{1}
	\end{aligned}
	\label{eq:equa3}
	\end{equation}
	where $x_{0}$ is the leader state vector, which could be time varying;
	$x_{0,m_{p}}\in\mathbb{R}^{P}$ is the $m^{th}$ state variable of
	the leader where $x_{0}=\left[x_{0}^{1},\ldots,x_{0}^{M_{P}}\right]^{\top}$;
	the leader nonlinear function vector $f_{0}:\left[0,\infty\right)\times\mathbb{R}^{PM_{P}}\rightarrow\mathbb{R}^{P}$
	is piece-wise continuous in $t$ and locally Lipschitz. The disagreement
	variable in $i$ is $\delta_{i}^{1}=x_{i}^{1}-x_{0}^{1}$ and the
	global disagreement is 
	\begin{equation}
	\begin{aligned}\gamma^{M_{p}} & =x^{1}-\underline{x}_{0}^{1}\end{aligned}
	\label{eq:equa4}
	\end{equation}
	where $\gamma^{M_{p}}=\left[\gamma_{1}^{1},\ldots,\gamma_{N}^{1}\right]^{\top}\in\mathbb{R}^{PN}$,
	$\underline{x}_{0}^{1}=\left[x_{0}^{1},\ldots,x_{0}^{1}\right]^{\top}\in\mathbb{R}^{PN}$.
	It is assumed that the distributed state information of the communication
	graph for $i$th node is known. The neighborhood synchronization error
	$e:=e\left(t\right)$ is defined as in \cite{li_pinning_2004,khoo_robust_2009}
	by 
	\begin{equation}
	\begin{aligned}e_{i} & =\sum_{j\in N_{i}}a_{i,j}\left(x_{i}^{1}-x_{j}^{1}\right)+b_{i,i}\left(x_{i}^{1}-x_{0}^{1}\right)\end{aligned}
	\label{eq:equa5}
	\end{equation}
	where $e_{i}=\left[e_{i}^{1},\ldots,e_{i}^{P}\right]^{\top}\in\mathbb{R}^{P}$,
	$a_{i,j}\geq0$ and $a_{i,j}>0$ when agent $i$ is a neighbor of
	agent $j$; $b_{i}\geq0$ and $b_{i}>0$ when one or more agents are
	neighbors of the leader. $e^{p}=\left[e_{1}^{p},\ldots,e_{N}^{p}\right]^{\top}\in\mathbb{R}^{N}$,
	$p=1,\ldots,P$ and $B={\rm diag}\{b_{i}\}\in\mathbb{R}^{N\times N}$.
	Hence, the global error dynamics in \eqref{eq:equa6} 
	\begin{equation}
	\begin{aligned}e & =-\left(L+B\right)\left(\underline{x}_{0}^{1}-x^{1}\right)=\left(L+B\right)\left(x^{1}-\underline{x}_{0}^{1}\right)\end{aligned}
	\label{eq:equa6}
	\end{equation}
	By re-writing global error dynamics in state vector form gives the
	expression in \eqref{eq:equa7} 
	\begin{equation}
	\begin{aligned}\dot{e}^{1} & =e^{2}\\
	\dot{e}^{2} & =e^{3}\\
	& \vdots\\
	\dot{e}^{M_{p}} & =\left(L+B\right)\left(f\left(x\right)+Gu-\underline{f}_{0}\right)
	\end{aligned}
	\label{eq:equa7}
	\end{equation}
	with bounded reference nonlinear dynamics\\
	$\underline{f}_{0}=\left[f_{0}\left(t,x_{0}\right),\ldots,f_{0}\left(t,x_{0}\right)\right]^{\top}\in\mathbb{R}^{N}$.
	The proof of equation \eqref{eq:equa7} can be found in \cite{bechlioulis_robust_2008}.\\
	
	\begin{rem}
		\label{Rem1}The global error dynamics in \eqref{eq:equa6} for MIMO
		where $P>1$ becomes 
		\begin{equation}
		\begin{aligned}e & =-\left(\left(L+B\right)\otimes\mathbf{I}_{P}\right)\left(\underline{x}_{0}^{1}-x^{1}\right)\\
		& =\left(\left(L+B\right)\otimes\mathbf{I}_{P}\right)\left(x^{1}-\underline{x}_{0}^{1}\right)
		\end{aligned}
		\label{eq:equa8}
		\end{equation}
		In the same vein as \eqref{eq:equa8}, \eqref{eq:equa7} becomes 
		\begin{equation}
		\begin{aligned}\dot{e}^{1} & =e^{2}\\
		\dot{e}^{2} & =e^{3}\\
		& \vdots\\
		\dot{e}^{M_{p}} & =\left(\left(L+B\right)\otimes\mathbf{I}_{P}\right)\left(f\left(x\right)+Gu-\underline{f}_{0}\right)
		\end{aligned}
		\label{eq:equa9}
		\end{equation}
		where $\otimes$ is the Kronecker product and $\mathbf{I}_{P}\in\mathbb{R}^{P\times P}$
		is the identity matrix. 
	\end{rem}
	\begin{rem}
		\label{Rem2} If $b_{i}\neq0$ for at least one $i$ with $i=1,\ldots,N$
		then $\left(L+B\right)$ is an irreducible diagonally dominant matrix
		M and hence nonsingular \cite{Qu2009}. 
	\end{rem}
	For a strongly connected graph, $B\neq0$ and 
	\begin{equation}
	\begin{aligned}\left\Vert e_{0}\right\Vert  & \leq\frac{\left\Vert e\right\Vert }{\underline{\sigma}\left(L+B\right)}\end{aligned}
	\label{eq:equa10}
	\end{equation}
	where $\underline{\sigma}\left(L+B\right)$ is the minimum singular
	value of $L+B$.
	
	\section{Prescribed Performance \label{sec5}}
	
	As mentioned earlier, error can be confined using a decreasing function
	with prescribed features known as the Prescribed Performance Function
	(PPF) \cite{bechlioulis_robust_2008}. PPF offers a way of improving
	the transient behavior and control signal by enforcing certain features
	on error signal. $\rho\left(t\right)$ denotes the main component
	of PPF which is a smooth positive decreasing function given in \eqref{eq:equa11}
	such that $\rho:\mathbb{R}_{+}\to\mathbb{R}_{+}$ and $\lim\limits _{t\to\infty}\rho\left(t\right)=\rho_{\infty}>0$
	hold. 
	\begin{equation}
	\rho_{i}^{p}\left(t\right)=\left(\rho_{i,0}^{p}-\rho_{i,\infty}^{p}\right)\exp\left(-\ell_{i}^{p}t\right)+\rho_{i,\infty}^{p}\label{eq:equa11}
	\end{equation}
	where $\rho_{i,0}^{p}$ is the maximum or the initial value of PPF
	set, $\rho_{i,\infty}^{p}$ is the minimum or the steady state value
	of allocated set, and $\ell_{i}^{p}$ tunes the reduction of set boundaries
	with respect to time. PPF and error component $e\left(t\right)$ satisfy
	the following properties: 
	\begin{align}
	-\delta_{i}^{p}\rho_{i}^{p}\left(t\right)<e_{i}^{p}\left(t\right)<\rho_{i}^{p}\left(t\right) & ,\text{ if }e_{i}^{p}\left(0\right)>0\label{eq:equa12}\\
	-\rho_{i}^{p}\left(t\right)<e_{i}^{p}\left(t\right)<\delta_{i}^{p}\rho_{i}^{p}\left(t\right) & ,\text{ if }e_{i}^{p}\left(0\right)<0\label{eq:equa13}
	\end{align}
	for all $t\geq0$, $1\geq\delta_{i}^{p}\geq0$, $i=1,\ldots,N$ and
	$p=1,\ldots,P$. The tracking errors of agents with prescribed performance
	reduces from a large to a smaller set in accordance with \eqref{eq:equa12}
	and \eqref{eq:equa13} as shown in Fig.~\ref{fig:fig1}. The complete
	idea of prescribed performance is illustrated in Fig.~\ref{fig:fig1}.
	\begin{figure}[h!]
		\centering \includegraphics[scale=0.27]{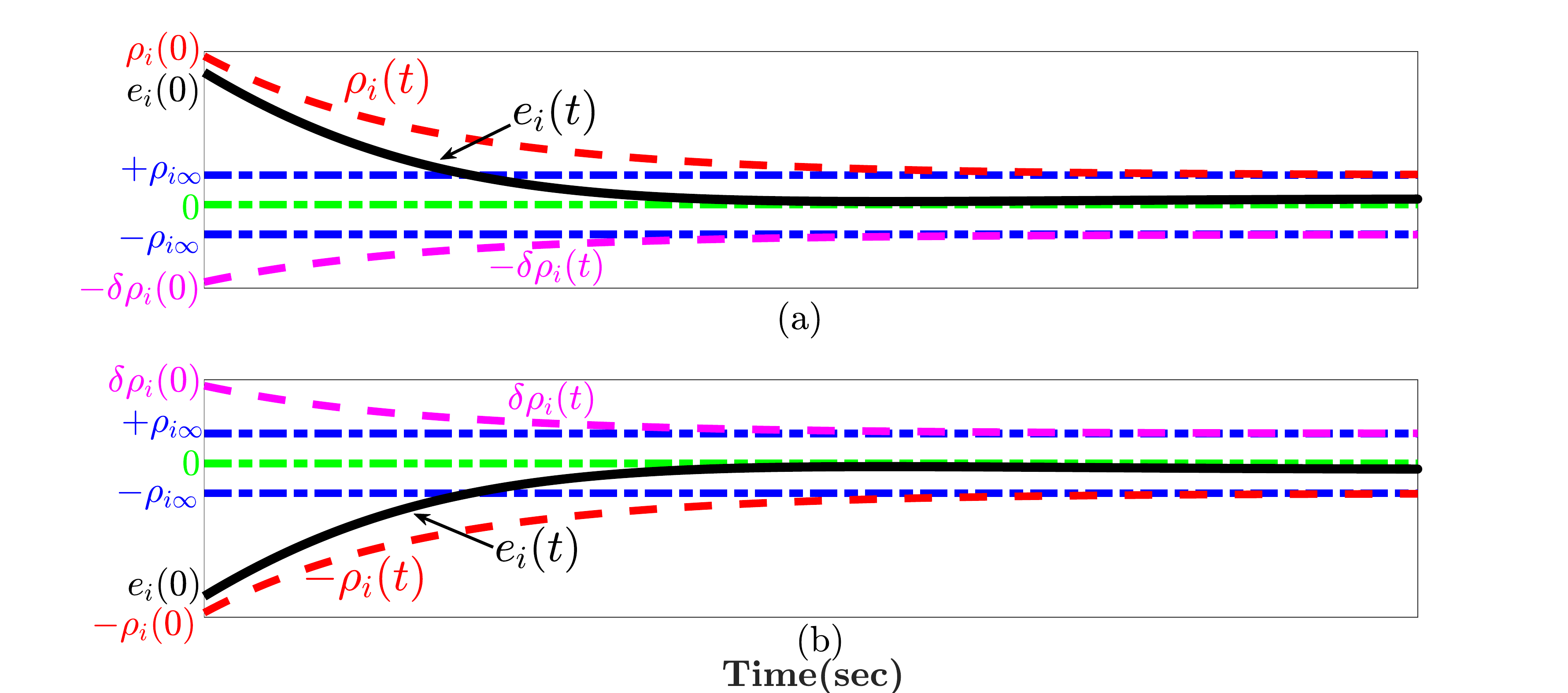} \caption{ Graphical representation of tracking error with prescribed performance
			satisfying (a) Eq.~\eqref{eq:equa12}; (b) Eq.~\eqref{eq:equa13}.}
		\label{fig:fig1} 
	\end{figure}
	
	\begin{rem}
		\label{rem3} As detailed in \cite{bechlioulis_robust_2008,hashim2017adaptive,hashim2017neuro},
		knowing the sign of $e_{i}\left(0\right)$ is sufficient to maintain
		the same robust controller for all $t>0$ and satisfy the performance
		constraints since switching does not occur after $t=0$. 
	\end{rem}
	A transformed error drives the error dynamics from constrained bounds
	in either \eqref{eq:equa12} or \eqref{eq:equa13} to that in \eqref{eq:equa14}
	which is unconstrained. 
	\begin{equation}
	\varepsilon_{i}^{p}\left(t\right)=\Upsilon\left(\frac{e_{i}^{p}\left(t\right)}{\rho_{i}^{p}\left(t\right)}\right)\label{eq:equa14}
	\end{equation}
	Subsequently, 
	\begin{equation}
	e_{i}^{p}\left(t\right)=\rho_{i}^{p}\left(t\right)\mathcal{F}\left(\varepsilon_{i}^{p}\right)\label{eq:equa15}
	\end{equation}
	where $\varepsilon_{i}^{p}$, $\mathcal{F}\left(\cdot\right)$ and
	$\Upsilon^{-1}\left(\cdot\right)$ are smooth functions, $i=1,2,\ldots,N$.
	For clarity, $\varepsilon:=\varepsilon\left(t\right)$ and $\rho:=\rho\left(t\right)$
	are defined accordingly. The smooth function $\mathcal{F}\left(\cdot\right)=\Upsilon^{-1}\left(\cdot\right)$
	and $\mathcal{F}\left(\cdot\right)$ must satisfy \cite{bechlioulis_robust_2008}: 
	\begin{enumerate}
		\item[P 1)] $\mathcal{F}\left(\varepsilon_{i}^{p}\right)$ is smooth and strictly
		increasing. 
		\item[P 2)] $\mathcal{F}\left(\varepsilon_{i}^{p}\right)$ is bounded between
		two predefined bounds \\
		$-\underline{\delta}_{i}^{p}<\mathcal{F}\left(\varepsilon_{i}^{p}\right)<\bar{\delta}_{i}^{p},\text{ if }e_{i}^{p}\left(0\right)\geq0$\\
		$-\bar{\delta}_{i}^{p}<\mathcal{F}\left(\varepsilon_{i}^{p}\right)<\underline{\delta}_{i}^{p},\text{ if }e_{i}^{p}\left(0\right)<0$ 
		\item[P 3)] 
		\item[] $\left.\begin{array}{c}
		\underset{\varepsilon_{i}^{p}\rightarrow-\infty}{\lim}\mathcal{F}\left(\varepsilon_{i}^{p}\right)=-\underline{\delta}_{i}^{p}\\
		\underset{\varepsilon_{i}^{p}\rightarrow+\infty}{\lim}\mathcal{F}\left(\varepsilon_{i}^{p}\right)=\bar{\delta}_{i}^{p}
		\end{array}\right\} \text{ if }e_{i}^{p}\left(0\right)\geq0$\\
		$\left.\begin{array}{c}
		\underset{\varepsilon_{i}^{p}\rightarrow-\infty}{\lim}\mathcal{F}\left(\varepsilon_{i}^{p}\right)=-\bar{\delta}_{i}^{p}\\
		\underset{\varepsilon_{i}^{p}\rightarrow+\infty}{\lim}\mathcal{F}\left(\varepsilon_{i}^{p}\right)=\underline{\delta}_{i}^{p}
		\end{array}\right\} \text{ if }e_{i}^{p}\left(0\right)<0$ 
	\end{enumerate}
	where\\
	
	\begin{equation}
	\mathcal{F}\left(\varepsilon_{i}^{p}\right)=\begin{cases}
	\frac{\bar{\delta}_{i}^{p}\exp\left(\varepsilon_{i}^{p}\right)-\underline{\delta}_{i}^{p}\exp\left(-\varepsilon_{i}^{p}\right)}{\exp\left(\varepsilon_{i}^{p}\right)+\exp\left(-\varepsilon_{i}^{p}\right)}, & \bar{\delta}_{i}^{p}\geq\underline{\delta}_{i}^{p}\text{ if }e_{i}^{p}\left(0\right)\geq0\\
	\frac{\bar{\delta}_{i}^{p}\exp\left(\varepsilon_{i}^{p}\right)-\underline{\delta}_{i}^{p}\exp\left(-\varepsilon_{i}^{p}\right)}{\exp\left(\varepsilon_{i}^{p}\right)+\exp\left(-\varepsilon_{i}^{p}\right)}, & \underline{\delta}_{i}^{p}\geq\bar{\delta}_{i}^{p}\text{ if }e_{i}^{p}\left(0\right)<0
	\end{cases}\label{eq:equa16}
	\end{equation}
	The function $\mathcal{F}\left(\varepsilon_{i}\right)$ is given as
	\begin{equation}
	\begin{aligned}\mathcal{F}\left(\varepsilon_{i}^{p}\right)= & \frac{\bar{\delta}_{i}^{p}\exp_{i}^{p}\left(\varepsilon_{i}^{p}\right)-\underline{\delta}_{i}^{p}\exp_{i}^{p}\left(-\varepsilon_{i}^{p}\right)}{\exp_{i}^{p}\left(\varepsilon_{i}^{p}\right)+\exp_{i}^{p}\left(-\varepsilon_{i}^{p}\right)}\end{aligned}
	\label{eq:equa17}
	\end{equation}
	and the transformed error is defined as 
	\begin{equation}
	\begin{aligned}\varepsilon_{i}^{p}= & \mathcal{F}^{-1}\left(e_{i}^{p}/\rho_{i}^{p}\right)\\
	= & \frac{1}{2}\begin{cases}
	\text{ln}\frac{\underline{\delta}_{i}^{p}+e_{i}^{p}/\rho_{i}^{p}}{\bar{\delta}_{i}^{p}-e_{i}^{p}/\rho_{i}^{p}}, & \bar{\delta}_{i}^{p}\geq\underline{\delta}_{i}^{p}\text{ if }e_{i}^{p}\left(0\right)\geq0\\
	\text{ln}\frac{\underline{\delta}_{i}^{p}+e_{i}/\rho_{i}^{p}}{\bar{\delta}_{i}^{p}-e_{i}^{p}/\rho_{i}^{p}}, & \underline{\delta}_{i}^{p}\geq\bar{\delta}_{i}^{p}\text{ if }e_{i}^{p}\left(0\right)<0
	\end{cases}
	\end{aligned}
	\label{eq:equa18}
	\end{equation}
	Then $\dot{\varepsilon}_{i}^{p}$ is governed by 
	\begin{equation}
	\dot{\varepsilon}_{i}^{p}=\frac{1}{2\rho_{i}^{p}}\left(\frac{1}{\underline{\delta}_{i}^{p}+e_{i}^{p}/\rho_{i}^{p}}+\frac{1}{\bar{\delta}_{i}^{p}-e_{i}^{p}/\rho_{i}^{p}}\right)\left(\dot{e}_{i}^{p}-\frac{e_{i}^{p}\dot{\rho}_{i}^{p}}{\rho_{i}^{p}}\right)\label{eq:equa19}
	\end{equation}
	where $\varepsilon_{i}\in\mathbb{R}^{P\times1}$.\\
	From \eqref{eq:equa19}, new variable $r_{i}^{p}$ is defined as
	follows 
	\begin{equation}
	\begin{split}r_{i}^{p} & =\frac{1}{2\rho_{i}^{p}}\frac{\partial\mathcal{F}^{-1}\left(e_{i}^{p}/\rho_{i}^{p}\right)}{\partial\left(e_{i}^{p}/\rho_{i}^{p}\right)}\\
	& =\frac{1}{2\rho_{i}^{p}}\left(\frac{1}{\underline{\delta}_{i}^{p}+e_{i}^{p}/\rho_{i}^{p}}+\frac{1}{\bar{\delta}_{i}^{p}-e_{i}^{p}/\rho_{i}^{p}}\right)
	\end{split}
	\label{eq:equa20}
	\end{equation}
	Then a metric error ${\bf E}_{i}\in\mathbb{R}^{P\times1}$ is introduced
	and it is given as in \eqref{eq:equa21} or equivalently in \eqref{eq:equa22}.
	\begin{equation}
	{\bf E}_{i}=\left(\frac{d}{dt}+\lambda_{i}^{m_{p}}\right)^{M_{p}-1}\varepsilon_{i}^{1}\label{eq:equa21}
	\end{equation}
	\begin{equation}
	{\bf E}_{i}=\varepsilon_{i}^{M_{p}}+\lambda_{i,M_{p}-1}\varepsilon_{i,M_{p}-1}+\cdots+\lambda_{i}^{1}\varepsilon_{i}^{1}\label{eq:equa22}
	\end{equation}
	for $i=1,\ldots,N$ and $m_{p}=1,\ldots,M_{p}$. Putting \eqref{eq:equa22}
	in global form, we have: 
	\begin{equation}
	{\bf E}=\varepsilon^{M_{p}}+\lambda_{M_{p}-1}\varepsilon^{M_{p}-1}+\cdots+\lambda_{1}\varepsilon^{1}\label{eq:equa23}
	\end{equation}
	where $\lambda_{i}^{m_{p}}$ is a positive constant, $\varepsilon^{m_{p}}=\left[\varepsilon_{1}^{m_{p}},\ldots,\varepsilon_{N}^{m_{p}}\right]^{\top}$,
	$m_{p}=1,\ldots,M_{P}$.\\
	With the assumptions that: 
	
	\begin{align}
	\Phi_{1} & =\left[\varepsilon^{1},\varepsilon^{2},\ldots,\varepsilon^{Mp-1}\right]^{\top}\label{eq:equa500}\\
	\Phi_{2} & =\dot{\Phi}_{1}=\left[\varepsilon^{2},\varepsilon^{3},\ldots,\varepsilon^{Mp}\right]^{\top}\label{eq:equa501}\\
	l & =\left[0,0,\ldots,0,1\right]^{\top}\in\mathbb{R}^{M_{p}-1}\nonumber 
	\end{align}
	
	and 
	\[
	\begin{aligned}\Lambda & =\begin{bmatrix}0 & 1 & 0 & \cdots & 0 & 0\\
	0 & 0 & 1 & \cdots & 0 & 0\\
	\vdots & \vdots & \vdots & \ddots & \vdots & \vdots\\
	0 & 0 & 0 & \cdots & 0 & 1\\
	-\lambda^{1} & -\lambda^{2} & -\lambda^{3} & \cdots & \lambda^{M_{p}-2} & -\lambda^{M_{p}-1}
	\end{bmatrix}\end{aligned}
	\]
	and $\Lambda$ is Hurwitz, then, 
	\begin{equation}
	\begin{aligned}\Phi_{2} & =\Phi_{1}\Lambda^{\top}+{\bf E}l^{\top}\end{aligned}
	\label{eq:equa24}
	\end{equation}
	and 
	\begin{equation}
	\begin{aligned}\Lambda^{\top}M+M\Lambda= & -\beta\mathbf{I}_{M_{p}-1}\end{aligned}
	\label{eq:equa25}
	\end{equation}
	where $\beta$ is a positive constant, $M>0$ and $\mathbf{I}_{M_{p}-1}$
	is the identity matrix with dimension $M_{p}-1$. By considering \eqref{eq:equa7}
	and \eqref{eq:equa19}, then, the derivative of the metric error in
	\eqref{eq:equa21} is given as 
	\begin{equation}
	\begin{aligned}\dot{{\bf E}}_{i}= & \sum\limits _{j=1}^{M_{p}-1}\begin{bmatrix}M_{p}-1\\
	j
	\end{bmatrix}\lambda_{i}^{j}\varepsilon_{i}^{M_{p}-j}+\varepsilon_{i}^{M_{p}}\end{aligned}
	\label{eq:equa26}
	\end{equation}
	with $\bar{\lambda}=\left[\lambda^{1},\ldots,\lambda^{M_{p}-1}\right]^{\top}\in\mathbb{R}^{PN}$.
	Hence, the global form of \eqref{eq:equa26} is 
	\begin{equation}
	\begin{aligned}\dot{{\bf E}} & =\varepsilon^{M_{p}+1}+\Phi_{2}\bar{\lambda}\\
	& =R\left(L+B\right)\left(f\left(x\right)+Gu-\underline{f}_{0}\right)+\Delta+\Phi_{2}\bar{\lambda}
	\end{aligned}
	\label{eq:equa27}
	\end{equation}
	where ${\bf E}=\left[{\bf E}_{1},\ldots,{\bf E}_{N}\right]^{\top}\in\mathbb{R}^{PN}$;
	$\varepsilon^{M_{p}+1}=e^{M_{p}+1}+\Delta$; $\Delta$ is the function
	of higher orders of $\rho_{i}^{p}$, $r_{i}^{p}$. $\Delta$ is vanishing
	because high orders of $\rho_{i}^{p}$, $r_{i}^{p}$ vanish, that
	is, tend to to zero as $t\rightarrow\infty$. Also, $R={\rm diag}\left\{ \Omega_{i}\right\} \in\mathbb{R}^{PN\times PN}$
	and 
	\[
	\Omega_{i}=\begin{bmatrix}\frac{1}{2\rho_{i}^{1}}\frac{\partial\mathcal{F}^{-1}\left(e_{i}^{1}/\rho_{i}^{1}\right)}{\partial\left(e_{i}^{1}/\rho_{i}^{1}\right)} & \cdots & 0\\
	\vdots & \ddots & \vdots\\
	0 & \cdots & \frac{1}{2\rho_{i}^{P}}\frac{\partial\mathcal{F}^{-1}\left(e_{i}^{P}/\rho_{i}^{P}\right)}{\partial\left(e_{i}^{P}/\rho_{i}^{P}\right)}
	\end{bmatrix}
	\]
	$\Omega_{i}$ is a decreasing positive definite matrix with $\Omega_{i}>0$
	and its components are defined in \eqref{eq:equa20}. The following
	definitions are also crucial (see \cite{das_distributed_2010}). 
	\begin{defn}
		\label{def1} The global neighborhood error $e\left(t\right)\in\mathbb{R}^{PN}$
		is uniformly ultimately bounded (UUB) if there exists a compact set
		$\Psi\subset\mathbb{R}^{PN}$ such that $\forall e\left(t_{0}\right)\in\Psi$,
		there exist a bound $B$ and a time $t_{f}\left(B,e\left(t_{0}\right)\right)$,
		both independent at $t_{0}\geq0$, so that $\left\Vert e\left(t\right)\right\Vert \leq B$
		$\forall t>t_{0}+t_{f}$. 
	\end{defn}
	\begin{defn}
		\label{def2} The control node trajectory $x_{0}\left(t\right)$ given
		in \eqref{eq:equa1} is cooperatively UUB with respect to solutions
		of node dynamics in \eqref{eq:equa3} if there exists a compact set
		$\Psi\subset\mathbb{R}^{PN}$ such that $\forall\left(x_{i}\left(t_{0}\right)-x_{0}\left(t_{0}\right)\right)\in\Psi$,
		there exist a bound $B$ and a time $t_{f}\left(B,\left(x\left(t_{0}\right)-x_{0}\left(t_{0}\right)\right)\right)$,
		both independent at $t_{0}\geq0$, so that $\left\Vert x\left(t_{0}\right)-x_{0}\left(t_{0}\right)\right\Vert \leq B$,
		$\forall i$, $\forall t>t_{0}+t_{f}$. 
	\end{defn}
	
	\section{Adaptive Projection Approximation \label{sec6}}
	
	The $i^{th}$ agent nonlinear dynamics in \eqref{eq:equa1} can be
	approximated as 
	\begin{equation}
	\begin{aligned}\dot{x}_{i}^{1} & =x_{i}^{2}\\
	\dot{x}_{i}^{2} & =x_{i}^{3}\\
	& \vdots\\
	\dot{x}_{i}^{M_{p}} & =G_{i}u_{i}+\theta_{i}\left\Vert x_{i}\right\Vert _{\infty}+\omega_{i}\\
	y_{i} & =x_{i}^{1}
	\end{aligned}
	\label{eq:equa28}
	\end{equation}
	with $\theta_{i},\omega_{i}\in\mathbb{R}^{P}$ are unknown entities
	in the model such as uncertainties and external disturbances \cite{hashim2017adaptive}.                        
	\begin{assum}
		\label{assump1} (Uniform boundedness of the unknown parameters)\\
		Let $\theta_{i}\in\Theta_{i}$ and $\omega_{i}\in\Delta_{i}$ of
		which $\Theta_{i}$ and $\Delta_{i}$ are known convex compact sets. 
	\end{assum}
	The unknown nonlinearities of a local node can be approximated by
	\begin{equation}
	\begin{aligned}f\left(x_{i}\right) & =\theta_{i}\left\Vert x_{i}\right\Vert {}_{\infty}+\omega_{i}\left(x_{i},t\right)\end{aligned}
	\label{eq:equa29}
	\end{equation}
	
	\begin{assum}
		\label{assump2} Matrix $G_{i}$ is known and invertible, that is,
		$G_{i}^{-1}$ exists. 
	\end{assum}
	Let $\hat{\theta}$ and $\hat{\omega}$ be the approximations of $\theta$
	and $\omega$, respectively. Then the estimation of unknown nonlinearities
	for a local node is given as 
	\begin{equation}
	\begin{aligned}\hat{f}\left(x_{i}\right) & =\hat{\theta}_{i}\left\Vert x_{i}\right\Vert _{\infty}+\hat{\omega}_{i}\left(x_{i},t\right)\end{aligned}
	\label{eq:equa30}
	\end{equation}
	The estimation error of nonlinearities is thus 
	\begin{equation}
	\begin{aligned}\tilde{f}\left(x_{i}\right) & =f\left(x_{i}\right)-\hat{f}\left(x_{i}\right)=\tilde{\theta_{i}}\left\Vert x_{i}\right\Vert _{\infty}+\tilde{\omega}_{i}\end{aligned}
	\label{eq:equa31}
	\end{equation}
	where 
	\begin{equation}
	\tilde{\theta}_{i}=\theta_{i}-\hat{\theta}_{i}\label{eq:equa32}
	\end{equation}
	\begin{equation}
	\tilde{\omega}_{i}=\omega_{i}-\hat{\omega}_{i}\label{eq:equa33}
	\end{equation}
	In this work, subsequent assumptions are considered. It is worthy
	of note that the values of the estimation bounds are not necessary
	known. 
	\begin{assum}
		\label{assump3} \cite{hovakimyan_l1_2010}                                           
	\begin{enumerate}
		\item States of the leader are bounded by $\left\Vert x_{0}\right\Vert \leq X_{0}$. 
		\item Dynamics of the leader is unknown and bounded such that $\left\Vert \underline{f}_{0}\left(x_{0},t\right)\right\Vert \leq F_{M}$. 
		\item Unknown parameters are uniformly bounded: $\left\Vert \theta\right\Vert \leq\theta_{M}$
		and $\left\Vert \omega\right\Vert \leq\omega_{M}$ for all $t>0$. 
	\end{enumerate}
	\end{assum}
	\begin{lem}
		\label{Lemma1}\cite{lewis_cooperative_2013} Define 
		\begin{equation}
		q=\left[q_{1},\ldots,q_{N}\right]^{\top}=\left(L+B\right)^{\top}\cdot{\bf \underline{1}}_{N}\label{eq:equa34}
		\end{equation}
		\begin{equation}
		\mathcal{M}={\rm diag}\left\{ m_{i}\right\} ={\rm diag}\left\{ 1/q_{i}\right\} \label{eq:equa35}
		\end{equation}
		Then $\mathcal{M}>0$ and the matrix $Q$ is defined as 
		\begin{equation}
		Q=\mathcal{M}\left(L+B\right)+\left(L+B\right)^{\top}\mathcal{M}\label{eq:equa36}
		\end{equation}
	\end{lem}
	Let the control signal of local node $i$ be given as 
	\begin{equation}
	\begin{aligned}u_{i}= & -G_{i}^{-1}\left(c{\bf E}_{i}+\hat{\theta}_{i}\left\Vert x_{i}\right\Vert _{\infty}+\hat{\omega}_{i}\right)\\
	& -G_{i}^{-1}(d_{i}+b_{i})^{-1}\Omega_{i}^{-1}\left(\lambda_{i,M_{p}-1}\varepsilon_{i}^{M_{p}}+\cdots+\lambda_{i}^{1}\varepsilon_{i}^{2}\right)
	\end{aligned}
	\label{eq:equa37}
	\end{equation}
	such that the global form of the control input in \eqref{eq:equa37}
	is defined as 
	\begin{equation}
	u=-G^{-1}\left(c{\bf E}+\hat{\theta}\left\Vert x\right\Vert _{\infty}+\hat{\omega}+\left(B+D\right)^{-1}R^{-1}\Phi_{2}\bar{\lambda}\right)\label{eq:equa38}
	\end{equation}
	where $\left\Vert x\right\Vert _{\infty}=\left[\left\Vert x_{1}\right\Vert _{\infty},\ldots,\left\Vert x_{N}\right\Vert _{\infty}\right]^{\top}\otimes{\bf \underline{1}}_{N\times1}$,
	${\bf \underline{1}}_{N\times1}=\left[1,\ldots,1\right]^{\top}\in\mathbb{R}^{N\times1}$
	and the control gain $c>0$ such that 
	\begin{equation}
	c>\frac{1}{\underline{\sigma}\left(Q\right)\underline{\sigma}\left(R\right)}\left(\frac{\gamma_{1}^{2}+\gamma_{2}^{2}}{k}+\frac{2}{\beta}g^{2}+\nu\right)\label{eq:equa41}
	\end{equation}
	and $Q$ is as defined in \eqref{eq:equa36}. The control variable
	$c$ satisfies $\gamma:=-\frac{1}{2}\Phi\bar{\sigma}\left(\mathcal{M}\right)\bar{\sigma}\left(R\right)\bar{\sigma}\left(A\right)$,\\
	$g:=-\frac{1}{2}\left(\bar{\sigma}\left(M\right)+\frac{\bar{\sigma}\left(\mathcal{M}\right)\bar{\sigma}\left(A\right)}{\underline{\sigma}(B+D)}\left\Vert \Lambda\right\Vert _{F}\left\Vert \bar{\lambda}\right\Vert \right)$,
	$\nu:=\frac{\bar{\sigma}\left(\mathcal{M}\right)\bar{\sigma}\left(A\right)}{\underline{\sigma}(B+D)}\left\Vert \bar{\lambda}\right\Vert $,
	and $M$ is as defined in \eqref{eq:equa25} for $\beta>0$. The adaptive
	estimates $\hat{\theta}$ and $\hat{\omega}$ are updated according
	to \eqref{eq:equa39} and \eqref{eq:equa40}. 
	\begin{equation}
	\dot{\hat{\theta}}_{i}=\Gamma_{1i}\left(d_{i}+b_{i}\right)r_{i}m_{i}{\bf E}_{i}\left\Vert x_{i}\right\Vert _{\infty}-k\Gamma_{1i}\hat{\theta}_{i}\label{eq:equa39}
	\end{equation}
	\begin{equation}
	\dot{\hat{\omega}}_{i}=\Gamma_{2i}\left(d_{i}+b_{i}\right)r_{i}m_{i}{\bf E}_{i}-k\Gamma_{2i}\hat{\omega}_{i}\label{eq:equa40}
	\end{equation}
	with $\Gamma_{1i},\Gamma_{2i}\in\mathbb{R}^{+}$ and $k>0$.
	\begin{thm}
		\label{theorem1} Let $L$ be an irreducible matrix and $B\neq0$
		such that $\left(L+B\right)$ is nonsingular. $c$ and $k$ are scalar
		design parameters defined in \eqref{eq:equa41}. Consider the distributed
		system in \eqref{eq:equa1} and the leader dynamics in \eqref{eq:equa3}
		coupled directly to the control input in \eqref{eq:equa37} with adaptive
		estimates \eqref{eq:equa39} and \eqref{eq:equa40}. If Assumptions
		\ref{assump2} and \ref{assump3} hold and the distributed control
		is as in \eqref{eq:equa38}. Then, the control node trajectory $x_{0}\left(t\right)$
		is cooperatively UUB and all nodes synchronize close to $x_{0}\left(t\right)$. 
	\end{thm}
	\begin{proof} Based on \eqref{eq:equa36} in Lemma \ref{Lemma1},
		the error function in \eqref{eq:equa7} then becomes 
		\[
		\dot{e}^{M_{p}}=\left(L+B\right)\left(\theta\left\Vert x\right\Vert _{\infty}+\omega+Gu-\underline{f}_{0}\left(x_{0},t\right)\right)
		\]
		and can be rewritten as 
		\begin{equation}
		\begin{aligned}\dot{e}=\left(L+B\right)\big( & \tilde{\theta}\left\Vert x\right\Vert _{\infty}+\tilde{\omega}-c{\bf E}-\left(B+D\right)^{-1}R^{-1}\Phi_{2}\bar{\lambda}\\
		& -\underline{f}\left(x_{0},t\right)\big)
		\end{aligned}
		\label{eq:equa42}
		\end{equation}
		and from \eqref{eq:equa27}, the transformed error then yields 
		\begin{equation}
		\begin{aligned}\dot{{\bf E}}=R\left(L+B\right)\big( & \tilde{\theta}\left\Vert x\right\Vert _{\infty}+\tilde{\omega}-c{\bf E}-\left(B+D\right)^{-1}R^{-1}\Phi_{2}\bar{\lambda}\\
		& -\underline{f}\left(x_{0},t\right)\big)+\Delta+\Phi_{2}\bar{\lambda}
		\end{aligned}
		\label{eq:equa43}
		\end{equation}
		By considering the following Lyapunov candidate function 
		\begin{equation}
		\begin{aligned}V & =\frac{1}{2}{\bf E}^{\top}\mathcal{M}{\bf E}+\frac{1}{2}\tilde{\theta}^{\top}\Gamma_{1}^{-1}\tilde{\theta}+\frac{1}{2}\tilde{\omega}^{\top}\Gamma_{2}^{-1}\tilde{\omega}+\frac{1}{2}{\rm Tr}\left\{ \Phi_{1}M\Phi_{1}^{\top}\right\} \\
		& =V_{1}+V_{2}+V_{3}+V_{4}
		\end{aligned}
		\label{eq:equa44}
		\end{equation}
		with $\mathcal{M}>0$ is defined in Lemma \ref{Lemma1}, $\Gamma_{1i},\Gamma_{2i}\in\mathbb{R}^{P\times P}$
		were mentioned in \eqref{eq:equa39} and \eqref{eq:equa40}, respectively,
		such that $\Gamma_{1}:={\rm diag}\left\{ \Gamma_{1i}\right\} $, $\Gamma_{2}:={\rm diag}\left\{ \Gamma_{2i}\right\} $
		and strictly positive. Let $V_{1}:=\frac{1}{2}{\bf E}^{\top}\mathcal{M}{\bf E}$,
		$V_{2}:=\frac{1}{2}\tilde{\theta}^{\top}\Gamma_{1}^{-1}\tilde{\theta}$,
		$V_{3}:=\frac{1}{2}\tilde{\omega}^{\top}\Gamma_{2}^{-1}\tilde{\omega}$
		and $V_{4}:=\frac{1}{2}{\rm Tr}\left\{ \Phi_{1}M\Phi_{1}^{\top}\right\} $
		be defined in \eqref{eq:equa44} and $\Gamma_{1}^{-1}$ and $\Gamma_{2}^{-1}$
		are block diagonal matrices as in \eqref{eq:equa39} and \eqref{eq:equa40},
		respectively. For simplicity, let $\Gamma_{1i}=\Gamma_{2i}$. After
		algebraic manipulations and substitution of \eqref{eq:equa39} and
		\eqref{eq:equa40}, 
		\begin{equation}
		\dot{V_{1}}+\dot{V_{2}}+\dot{V_{3}}={\bf E}^{\top}\mathcal{M}\dot{{\bf E}}+\tilde{\theta}^{\top}\Gamma^{-1}\dot{\tilde{\theta}}+\tilde{\omega}^{\top}\Gamma^{-1}\dot{\tilde{\omega}}\label{eq:equa45}
		\end{equation}
	\end{proof} 
	\begin{equation}
	\begin{split}\dot{V_{1}} & +\dot{V_{2}}+\dot{V_{3}}\\
	& =-c{\bf E}^{\top}\mathcal{M}R\left(L+B\right){\bf E}+{\bf E}^{\top}\mathcal{M}R\left(B+D\right)\tilde{\theta}\left\Vert x\right\Vert _{\infty}\\
	& +{\bf E}^{\top}\mathcal{M}R\left(B+D\right)\tilde{\omega}+{\bf E}^{\top}\mathcal{M}RA\left(B+D\right)^{-1}R^{-1}\Phi_{2}\bar{\lambda}\\
	& -{\bf E}^{\top}\mathcal{M}R\left(L+B\right)\underline{f}\left(x_{0},t\right)-{\bf E}^{\top}\mathcal{M}RA\tilde{\theta}\left\Vert x\right\Vert _{\infty}\\
	& -{\bf E}^{\top}\mathcal{M}RA\tilde{\omega}+{\bf E}^{\top}\mathcal{M}\Delta+\tilde{\theta}^{\top}\Gamma_{1}^{-1}\dot{\tilde{\theta}}+\tilde{\omega}^{\top}\Gamma_{2}^{-1}\dot{\tilde{\omega}}
	\end{split}
	\label{eq:equa46}
	\end{equation}
	since $\dot{\tilde{\theta}}=\dot{\theta}-\dot{\hat{\theta}}=-\dot{\hat{\theta}}$
	and $\dot{\tilde{\omega}}=\dot{\omega}-\dot{\hat{\omega}}=-\dot{\hat{\omega}}$,
	one can write \eqref{eq:equa46} as 
	\begin{equation}
	\begin{split}\dot{V_{1}} & +\dot{V_{2}}+\dot{V_{3}}\\
	& =-c{\bf E}^{\top}\mathcal{M}R\left(L+B\right){\bf E}+{\bf E}^{\top}\mathcal{M}R\left(B+D\right)\tilde{\theta}\left\Vert x\right\Vert _{\infty}\\
	& +{\bf E}^{\top}\mathcal{M}R\left(B+D\right)\tilde{\omega}+{\bf E}^{\top}\mathcal{M}RA\left(B+D\right)^{-1}R^{-1}\Phi_{2}\bar{\lambda}\\
	& -{\bf E}^{\top}\mathcal{M}R\left(L+B\right)\underline{f}\left(x_{0},t\right)+{\bf E}^{\top}\mathcal{M}\Delta-{\bf E}^{\top}\mathcal{M}RA\tilde{\theta}\left\Vert x\right\Vert _{\infty}\\
	& -\tilde{\theta}^{\top}\Gamma_{1}^{-1}\left(\Gamma_{1}\left(B+D\right)R\mathcal{M}{\bf E}\left\Vert x\right\Vert _{\infty}-k\Gamma_{1}\hat{\theta}\right)\\
	& -\tilde{\omega}^{\top}\Gamma_{2}^{-1}\left(\Gamma_{2}\left(B+D\right)R\mathcal{M}{\bf E}-k\Gamma_{2}\hat{\omega}\right)-{\bf E}^{\top}\mathcal{M}RA\tilde{\omega}
	\end{split}
	\label{eq:equa47}
	\end{equation}
	Note that $x^{\top}y={\rm Tr}\left\{ yx^{\top}\right\} $, $\forall x,y\in\mathbb{R}^{N}$,
	then $\dot{V_{1}}+\dot{V_{2}}+\dot{V_{3}}$ can be written as 
	\begin{equation}
	\begin{split}\dot{V_{1}}+\dot{V_{2}} & +\dot{V_{3}}=-c{\bf E}^{\top}RQ{\bf E}+{\bf E}^{\top}\mathcal{M}R\left(B+D\right)\tilde{\theta}\left\Vert x\right\Vert _{\infty}\\
	& -{\bf E}^{\top}\mathcal{M}R\left(L+B\right)\underline{f}\left(x_{0},t\right)+{\bf E}^{\top}\mathcal{M}R\left(B+D\right)\tilde{\omega}\\
	& +{\bf E}^{\top}\mathcal{M}RA\left(B+D\right)^{-1}R^{-1}\Phi_{2}\bar{\lambda}+{\bf E}^{\top}\mathcal{M}\Delta\\
	& -{\bf E}^{\top}\mathcal{M}RA\tilde{\theta}\left\Vert x\right\Vert _{\infty}-{\bf E}^{\top}\mathcal{M}R\left(B+D\right)\left\Vert x\right\Vert _{\infty}\tilde{\theta}\\
	& +k\hat{\theta}^{\top}\tilde{\theta}-{\bf E}^{\top}\mathcal{M}R\left(B+D\right)\tilde{\omega}+k\hat{\omega}^{\top}\tilde{\omega}\\
	& -{\bf E}^{\top}\mathcal{M}RA\tilde{\omega}
	\end{split}
	\label{eq:equa48}
	\end{equation}
	Also, $\tilde{\theta}=\theta-\hat{\theta}$ and $\tilde{\omega}=\omega-\hat{\omega}$
	and by substituting \eqref{eq:equa25} with \eqref{eq:equa48}, we
	end up with 
	\begin{equation}
	\begin{split}\dot{V_{1}}+\dot{V_{2}} & +\dot{V_{3}}=-c{\bf E}^{\top}RQ{\bf E}\bar{\lambda}+{\bf E}^{\top}\mathcal{M}\Delta\\
	& +{\bf E}^{\top}\mathcal{M}RA\left(B+D\right)^{-1}R^{-1}\Phi_{1}\Lambda^{\top}\\
	& +{\bf E}^{\top}\mathcal{M}RA\left(B+D\right)^{-1}R^{-1}{\bf E}l^{\top}\bar{\lambda}\\
	& -{\bf E}^{\top}\mathcal{M}R\left(L+B\right)\underline{f}\left(x_{0},t\right)\\
	& -{\bf E}^{\top}\mathcal{M}RA\tilde{\theta}\left\Vert x\right\Vert _{\infty}+k\theta^{\top}\tilde{\theta}-k\tilde{\theta}^{\top}\tilde{\theta}\\
	& +k\omega^{\top}\tilde{\omega}-k\tilde{\omega}^{\top}\tilde{\omega}-{\bf E}^{\top}\mathcal{M}RA\tilde{\omega}
	\end{split}
	\label{eq:equa49}
	\end{equation}
	then $\eqref{eq:equa49}$ can be rewritten in inequality form using
	the second norm as 
	\begin{equation}
	\begin{split}\dot{V_{1}} & +\dot{V_{2}}+\dot{V_{3}}\leq\\
	& -c\underline{\sigma}\left(R\right)\underline{\sigma}\left(Q\right)\left\Vert {\bf E}\right\Vert ^{2}+\frac{\bar{\sigma}\left(\mathcal{M}\right)\bar{\sigma}\left(A\right)}{\underline{\sigma}\left(B+D\right)}\left\Vert \Lambda\right\Vert _{F}\left\Vert \Phi_{1}\right\Vert \left\Vert {\bf E}\right\Vert \\
	& +\frac{\bar{\sigma}\left(\mathcal{M}\right)\bar{\sigma}\left(A\right)}{\underline{\sigma}\left(B+D\right)}\left\Vert \bar{\lambda}\right\Vert \left\Vert l\right\Vert \left\Vert {\bf E}\right\Vert ^{2}+\bar{\sigma}\left(\mathcal{M}\right)\bar{\sigma}\left(\Delta\right)\left\Vert {\bf E}\right\Vert \\
	& +\bar{\sigma}\left(\mathcal{M}\right)\bar{\sigma}\left(R\right)\bar{\sigma}\left(L+B\right)\underline{f}_{M}\left\Vert {\bf E}\right\Vert \\
	& +\bar{\sigma}\left(\mathcal{M}\right)\bar{\sigma}\left(R\right)\bar{\sigma}\left(A\right)x_{M}\left\Vert \tilde{\theta}\right\Vert \left\Vert {\bf E}\right\Vert +k\theta_{M}\left\Vert \tilde{\theta}\right\Vert -k\left\Vert \tilde{\theta}\right\Vert ^{2}\\
	& +k\omega_{M}\left\Vert \tilde{\omega}\right\Vert -k\left\Vert \tilde{\omega}\right\Vert ^{2}+\bar{\sigma}\left(\mathcal{M}\right)\bar{\sigma}\left(R\right)\bar{\sigma}\left(A\right)\left\Vert \tilde{\omega}\right\Vert \left\Vert {\bf E}\right\Vert 
	\end{split}
	\label{eq:equa51}
	\end{equation}
	Also, from \eqref{eq:equa500}, \eqref{eq:equa501} and \eqref{eq:equa44},
	the derivative of the fourth Lyapunov term $V_{4}$ is defined by
	\begin{equation}
	\begin{split}\dot{V_{4}}=\frac{1}{2}{\rm Tr}\left\{ \dot{\Phi_{1}}M\Phi_{1}^{\top}+\Phi_{1}M\dot{\Phi_{1}}\right\} ={\rm Tr}\left\{ \Phi_{2}M\Phi_{1}^{\top}\right\} \end{split}
	\label{eq:equa52}
	\end{equation}
	substituting \eqref{eq:equa24} in \eqref{eq:equa52} yields 
	\begin{equation}
	\begin{split}\dot{V_{4}} & ={\rm Tr}\left\{ \Phi_{1}\left(M\Lambda+\Lambda^{\top}M\right)\Phi_{1}^{\top}\right\} +{\rm Tr}\left\{ {\bf E}l^{\top}M\Phi_{1}^{\top}\right\} \\
	& =-\frac{1}{2}\beta{\rm Tr}\left\{ \Phi_{1}\Phi_{1}^{\top}\right\} +{\rm Tr}\left\{ {\bf E}l^{\top}M\Phi_{1}^{\top}\right\} 
	\end{split}
	\label{eq:equa54}
	\end{equation}
	Now, we put \eqref{eq:equa51} and \eqref{eq:equa54} in the complete
	form 
	\begin{equation}
	\begin{split}\dot{V}\leq & -c\underline{\sigma}\left(R\right)\underline{\sigma}\left(Q\right)\left\Vert {\bf E}\right\Vert ^{2}+\frac{\bar{\sigma}\left(\mathcal{M}\right)\bar{\sigma}\left(A\right)}{\underline{\sigma}\left(B+D\right)}\left\Vert \Lambda\right\Vert _{F}\left\Vert \Phi_{1}\right\Vert \left\Vert {\bf E}\right\Vert \\
	& +\frac{\bar{\sigma}\left(\mathcal{M}\right)\bar{\sigma}\left(A\right)}{\underline{\sigma}\left(B+D\right)}\left\Vert \bar{\lambda}\right\Vert \left\Vert l\right\Vert \left\Vert {\bf E}\right\Vert ^{2}+\bar{\sigma}\left(\mathcal{M}\right)\bar{\sigma}\left(\Delta\right)\left\Vert {\bf E}\right\Vert \\
	& +\bar{\sigma}\left(\mathcal{M}\right)\left(\bar{\sigma}\left(R\right)\bar{\sigma}\left(L+B\right)\underline{f}_{M}+\bar{\sigma}\left(R\right)\bar{\sigma}\left(A\right)x_{M}\left\Vert \tilde{\theta}\right\Vert \right)\left\Vert {\bf E}\right\Vert \\
	& +k\theta_{M}\left\Vert \tilde{\theta}\right\Vert -k\left\Vert \tilde{\theta}\right\Vert ^{2}+k\omega_{M}\left\Vert \tilde{\omega}\right\Vert -k\left\Vert \tilde{\omega}\right\Vert ^{2}\\
	& +\bar{\sigma}\left(\mathcal{M}\right)\bar{\sigma}\left(R\right)\bar{\sigma}\left(A\right)\left\Vert \tilde{\omega}\right\Vert \left\Vert {\bf E}\right\Vert -\frac{1}{2}\beta\left\Vert \Phi_{1}\right\Vert ^{2}\\
	& +\bar{M}\left\Vert l\right\Vert \left\Vert \Phi_{1}\right\Vert \left\Vert {\bf E}\right\Vert 
	\end{split}
	\label{eq:equa55}
	\end{equation}
	with $\left\Vert l\right\Vert =1$ 
	\begin{equation}
	\begin{split}\dot{V}\leq & -\left(c\underline{\sigma}\left(R\right)\underline{\sigma}\left(Q\right)-\frac{\bar{\sigma}\left(\mathcal{M}\right)\bar{\sigma}\left(A\right)}{\underline{\sigma}\left(B+D\right)}\left\Vert \bar{\lambda}\right\Vert \right)\left\Vert {\bf E}\right\Vert ^{2}\\
	& +\left(\bar{M}+\frac{\bar{\sigma}\left(\mathcal{M}\right)\bar{\sigma}\left(A\right)}{\underline{\sigma}\left(B+D\right)}\left\Vert \Lambda\right\Vert _{F}\right)\left\Vert \Phi_{1}\right\Vert \left\Vert {\bf E}\right\Vert \\
	& +\bar{\sigma}\left(\mathcal{M}\right)\left(\bar{\sigma}\left(\Delta\right)+\bar{\sigma}\left(R\right)\bar{\sigma}\left(L+B\right)\underline{f}_{M}\right)\left\Vert {\bf E}\right\Vert \\
	& +\bar{\sigma}\left(\mathcal{M}\right)\bar{\sigma}\left(R\right)\bar{\sigma}\left(A\right)x_{M}\left\Vert \tilde{\theta}\right\Vert \left\Vert {\bf E}\right\Vert \\
	& +k\theta_{M}\left\Vert \tilde{\theta}\right\Vert -k\left\Vert \tilde{\theta}\right\Vert ^{2}+k\omega_{M}\left\Vert \tilde{\omega}\right\Vert -k\left\Vert \tilde{\omega}\right\Vert ^{2}\\
	& +\bar{\sigma}\left(\mathcal{M}\right)\bar{\sigma}\left(R\right)\bar{\sigma}\left(A\right)\left\Vert \tilde{\omega}\right\Vert \left\Vert {\bf E}\right\Vert -\frac{1}{2}\beta\left\Vert \Phi_{1}\right\Vert ^{2}
	\end{split}
	\label{eq:equa56}
	\end{equation}
	Let us define
	
	\begin{align*}
	\gamma_{1} & =-\frac{1}{2}\bar{\sigma}\left(\mathcal{M}\right)\bar{\sigma}\left(R\right)\bar{\sigma}\left(A\right)x_{M}\\
	\gamma_{2} & =-\frac{1}{2}\bar{\sigma}\left(\mathcal{M}\right)\bar{\sigma}\left(R\right)\bar{\sigma}\left(A\right)\\
	g & =-\frac{1}{2}\left(\bar{\sigma}\left(M\right)+\frac{\bar{\sigma}\left(\mathcal{M}\right)\bar{\sigma}\left(A\right)}{\underline{\sigma}\left(B+D\right)}\left\Vert \Lambda\right\Vert _{F}\left\Vert \bar{\lambda}\right\Vert \right)\\
	\nu & =\frac{\bar{\sigma}\left(\mathcal{M}\right)\bar{\sigma}\left(A\right)}{\underline{\sigma}(B+D)}\left\Vert \bar{\lambda}\right\Vert \\
	\mu & =\left(c\underline{\sigma}\left(R\right)\underline{\sigma}\left(Q\right)-\frac{\bar{\sigma}\left(\mathcal{M}\right)\bar{\sigma}\left(A\right)}{\underline{\sigma}\left(B+D\right)}\right)
	\end{align*}
	hence, the inequality in \eqref{eq:equa56} can be expressed as
	\begin{equation}
	\begin{split}\dot{V}\leq & -\left[\begin{smallmatrix}\left\Vert \Phi_{1}\right\Vert  & \left\Vert \tilde{\theta}\right\Vert  & \left\Vert \tilde{\omega}\right\Vert  & \left\Vert {\bf E}\right\Vert \end{smallmatrix}\right]\left[\begin{smallmatrix}\frac{1}{2}\beta & 0 & 0 & g\\
	0 & k & 0 & \gamma_{1}\\
	0 & 0 & k & \gamma_{2}\\
	g & \gamma_{1} & \gamma_{2} & \mu
	\end{smallmatrix}\right]\left[\begin{smallmatrix}\left\Vert \Phi_{1}\right\Vert \\
	\left\Vert \tilde{\theta}\right\Vert \\
	\left\Vert \tilde{\omega}\right\Vert \\
	\left\Vert {\bf E}\right\Vert 
	\end{smallmatrix}\right]\\
	& +\left[\begin{smallmatrix}0 & k\theta_{M} & k\omega_{M} & \bar{\sigma}\left(\mathcal{M}\right)\left(\bar{\sigma}\left(\Delta\right)+\bar{\sigma}\left(R\right)\bar{\sigma}\left(L+B\right)\underline{f}_{M}\right)\end{smallmatrix}\right]\left[\begin{smallmatrix}\left\Vert \Phi_{1}\right\Vert \\
	\left\Vert \tilde{\theta}\right\Vert \\
	\left\Vert \tilde{\omega}\right\Vert \\
	\left\Vert {\bf E}\right\Vert 
	\end{smallmatrix}\right]
	\end{split}
	\label{eq:equa57}
	\end{equation}
	Define $z=\left[\begin{smallmatrix}\left\Vert \Phi_{1}\right\Vert  & \left\Vert \tilde{\theta}\right\Vert  & \left\Vert \tilde{\omega}\right\Vert  & \left\Vert {\bf E}\right\Vert \end{smallmatrix}\right]^{\top}$,
	$h=\left[\begin{smallmatrix}0 & k\theta_{M} & k\omega_{M} & \bar{\sigma}\left(\mathcal{M}\right)\left(\bar{\sigma}\left(\Delta\right)+\bar{\sigma}\left(R\right)\bar{\sigma}\left(L+B\right)\underline{f}_{M}\right)\end{smallmatrix}\right]^{\top}$
	and 
	\[
	\mathcal{H}=\begin{bmatrix}\frac{1}{2}\beta & 0 & 0 & g\\
	0 & k & 0 & \gamma_{1}\\
	0 & 0 & k & \gamma_{2}\\
	g & \gamma_{1} & \gamma_{2} & \mu
	\end{bmatrix}
	\]
	such that \eqref{eq:equa57} can be written in simpler form as 
	\begin{equation}
	\begin{split}\dot{V}\leq & -z^{\top}\mathcal{H}z+h^{\top}z\end{split}
	\label{eq:equa58}
	\end{equation}
	then, we have $\dot{V}\leq0$ if and only if $\mathcal{H}$ is positive
	definite and 
	\begin{equation}
	\begin{split}\left\Vert z\right\Vert > & \frac{\left\Vert h\right\Vert }{\underline{\sigma}(\mathcal{H})}\end{split}
	\label{eq:equa59}
	\end{equation}
	According to Sylvester's criterion, $\mathcal{H}>0$ if 
	\begin{enumerate}
		\item $\beta>0$ 
		\item $\beta k>0$ 
		\item $\beta k^{2}>0$ 
		\item $k(\beta\mu-2g^{2})-\beta(\gamma_{1}^{2}+\gamma_{2}^{2})>0$ 
	\end{enumerate}
	Solving the foregoing equations show \eqref{eq:equa41} 
	\[
	\begin{split}c>\frac{1}{\underline{\sigma}\left(Q\right)\underline{\sigma}\left(R\right)}\left(\frac{\gamma_{1}^{2}+\gamma_{2}^{2}}{k}+\frac{2}{\beta}g^{2}+\nu\right)\end{split}
	\]
	Then, with the assumption that 
	\begin{equation}
	\begin{split}\eta=\frac{k\theta_{M}+k\omega_{M}+\bar{\sigma}\left(\mathcal{M}\right)\left(\bar{\sigma}\left(\Delta\right)+\bar{\sigma}\left(R\right)\bar{\sigma}\left(L+B\right)\underline{f}_{M}\right)}{\underline{\sigma}\left(\mathcal{H}\right)}\end{split}
	\label{eq:equa60}
	\end{equation}
	we have $\dot{V}\leq0$ if $\left\Vert z\right\Vert >\eta$. In accordance
	with \eqref{eq:equa44}, we have 
	\begin{equation}
	\begin{aligned}\frac{1}{2}z^{\top} & \begin{bmatrix}\underline{\sigma}\left(M\right) & 0 & 0 & 0\\
	0 & \bar{\sigma}\left(\Gamma_{1}\right) & 0 & 0\\
	0 & 0 & \bar{\sigma}\left(\Gamma_{2}\right) & 0\\
	0 & 0 & 0 & \underline{\sigma}\left(\mathcal{M}\right)
	\end{bmatrix}z\leq V\leq\\
	& \frac{1}{2}z^{\top}\begin{bmatrix}\bar{\sigma}\left(M\right) & 0 & 0 & 0\\
	0 & \underline{\sigma}\left(\Gamma_{1}\right) & 0 & 0\\
	0 & 0 & \underline{\sigma}\left(\Gamma_{2}\right) & 0\\
	0 & 0 & 0 & \bar{\sigma}\left(\mathcal{M}\right)
	\end{bmatrix}z
	\end{aligned}
	\label{eq:equa61}
	\end{equation}
	By defining appropriate matching variables, \eqref{eq:equa61} becomes
	\begin{equation}
	\frac{1}{2}z^{\top}\underline{\mathcal{X}}z\leq V\leq\frac{1}{2}z^{\top}\bar{\mathcal{X}}z\label{eq:equa62}
	\end{equation}
	which is equivalent to 
	\begin{equation}
	\begin{split}\frac{1}{2}\underline{\sigma}\left(\underline{\mathcal{X}}\right)\left\Vert z\right\Vert ^{2}\leq V\leq\frac{1}{2}\bar{\sigma}\left(\bar{\mathcal{X}}\right)\left\Vert z\right\Vert ^{2}\end{split}
	\label{eq:equa63}
	\end{equation}
	then, 
	\begin{equation}
	\begin{split}V>\frac{1}{2}\bar{\sigma}\left(\bar{\mathcal{X}}\right)\frac{\left\Vert h\right\Vert ^{2}}{\underline{\sigma}^{2}(\mathcal{\mathcal{H}})}\end{split}
	\label{eq:equa64}
	\end{equation}
	Hence, based on Theorem 4.18 in \cite{khalil_nonlinear_2002}, for
	any the initial value $z\left(t_{0}\right)$ there exists $T_{0}$
	such that 
	\begin{equation}
	\begin{split}z\left(t\right)<\sqrt{\frac{\bar{\sigma}\left(\bar{\mathcal{X}}\right)}{\underline{\sigma}\left(\underline{\mathcal{X}}\right)}}\eta,\forall t\geq t_{0}+T_{0}\end{split}
	\label{eq:equa65}
	\end{equation}
	where time $T_{0}$ is evaluated using 
	\begin{equation}
	\begin{split}T_{0}=\frac{V\left(t_{0}\right)-\bar{\sigma}\left(\mathcal{X}\right)\eta^{2}}{k}\end{split}
	\label{eq:equa66}
	\end{equation}
	Also from \eqref{eq:equa66}, we have the relations 
	\begin{equation}
	\begin{split}\left\Vert z\right\Vert \leq\sqrt{\frac{2V}{\underline{\sigma}\left(\underline{\mathcal{X}}\right)}},\hspace{10pt}\left\Vert z\right\Vert \geq\sqrt{\frac{2V}{\bar{\sigma}\left(\bar{\mathcal{X}}\right)}}\end{split}
	\label{eq:equa67}
	\end{equation}
	Then, equation \eqref{eq:equa57} can be written as 
	\begin{equation}
	\dot{V}\leq-\tau_{1}V+\tau_{2}\sqrt{V}\label{eq:equa68}
	\end{equation}
	with $\tau_{1}:=\frac{2\underline{\sigma}\left(\mathcal{H}\right)}{\bar{\sigma}\left(\bar{\mathcal{X}}\right)}$
	and $\tau_{2}:=\frac{\sqrt{2}\left\Vert h\right\Vert }{\sqrt{\underline{\sigma}\left(\underline{\mathcal{X}}\right)}}$
	leading to 
	\begin{equation}
	\sqrt{V}\leq\sqrt{V\left(0\right)}+\frac{\tau_{2}}{\tau_{1}}\label{eq:equa69}
	\end{equation}
	Therefore $\varepsilon$ is $\mathcal{L}_{\infty}$ and contained
	for $t>t_{0}$ in the compact set $\Psi_{0}=\left\{ \varepsilon\left(t\right)\left|\left\Vert \varepsilon\left(t\right)\right\Vert \leq r_{t_{0}}\right.\right\} $
	as stated in Theorem \ref{theorem1}. Hence, $e\left(t\right)$ will
	satisfy the prescribed performance $\forall t$ if we start at $t=t_{0}$
	within the prescribed functions. 
	\begin{rem}
		\textbf{(Notes on control design parameters)} $\bar{\delta}_{i}$
		and $\underline{\delta}_{i}$ define the dynamic boundaries of the
		initial (large) and final (small) set and they have a significant
		impact on the control effort up to small error tracking which is bounded
		by $\rho_{i\infty}$. Higher values of $\bar{\delta}_{i}$ and $\underline{\delta}_{i}$
		require more time for the systematic convergence from large to small
		set. Whereas, the impact of $\ell_{i}$ can be noticed on the speed
		of convergence from large to small sets that implies values of $\ell_{i}$
		has a direct impact on the range of control signal. It should be remarked
		that $k$ and $\Gamma_{i}^{p}$ are associated with nonlinearity compensation
		of adaptive estimate, $c$ controls the speed of convergence to the
		desired tracking output and has to be selected to satisfy \eqref{eq:equa41}.
		The bounds on local consensus modified errors can be made small by
		increasing the control gains $c$. 
	\end{rem}
	Finally, the algorithm of nonlinear high order agent dynamics such
	as equation \eqref{eq:equa1} can be summarized briefly as 
	\begin{enumerate}
		\item[Step 1.] Select the control design parameters such as $\bar{\delta}_{i}^{p}$,
		$\underline{\delta}_{i}^{p}$, $\rho_{i,p\infty}$, $\ell_{i}^{p}$,
		$\Gamma_{i}^{p}$, $k$ and $c$. 
		\item[Step 2.] Evaluate local error synchronization $e_{i}$ from equation \eqref{eq:equa5}. 
		\item[Step 3.] Evaluate the prescribed performance function $\rho_{i}$ from equation
		\eqref{eq:equa11}. 
		\item[Step 4.] Evaluate $r_{i}^{p}$ from equation \eqref{eq:equa20}. 
		\item[Step 5.] Evaluate the transformed error from equation \eqref{eq:equa18} starting
		from $\varepsilon^{1}$ to $\varepsilon^{M_{p}}$. 
		\item[Step 6.] Evaluate the metric error ${\bf E}_{i}$ from \eqref{eq:equa21}
		or \eqref{eq:equa22}. 
		\item[Step 7.] Evaluate control signal $u_{i}$ from equation \eqref{eq:equa37}. 
		\item[Step 8.] Evaluate the adaptive estimates $\hat{\theta}_{i}$ and $\hat{\omega}_{i}$
		from equations \eqref{eq:equa39} and \eqref{eq:equa40}, respectively. 
		\item[Step 9.] Go to Step 2. 
	\end{enumerate}
	
	\section{Simulation Results \label{sec7}}
	
	In this section, we present two different examples to illustrate the
	robustness of the proposed controller. The first example considers
	single-input single-output (SISO) problem and the second example present
	multi-input multi-output (MIMO) problem with high order dynamics.
	Consider a connected network is composed of 5 agents denoted by 1
	to 5 with one leader denoted by 0 and the leader node is connected
	to agents 1 and 5 as in Fig. \ref{fig:fig2}. 
	\begin{figure}[h!]
		\centering \includegraphics[scale=0.5]{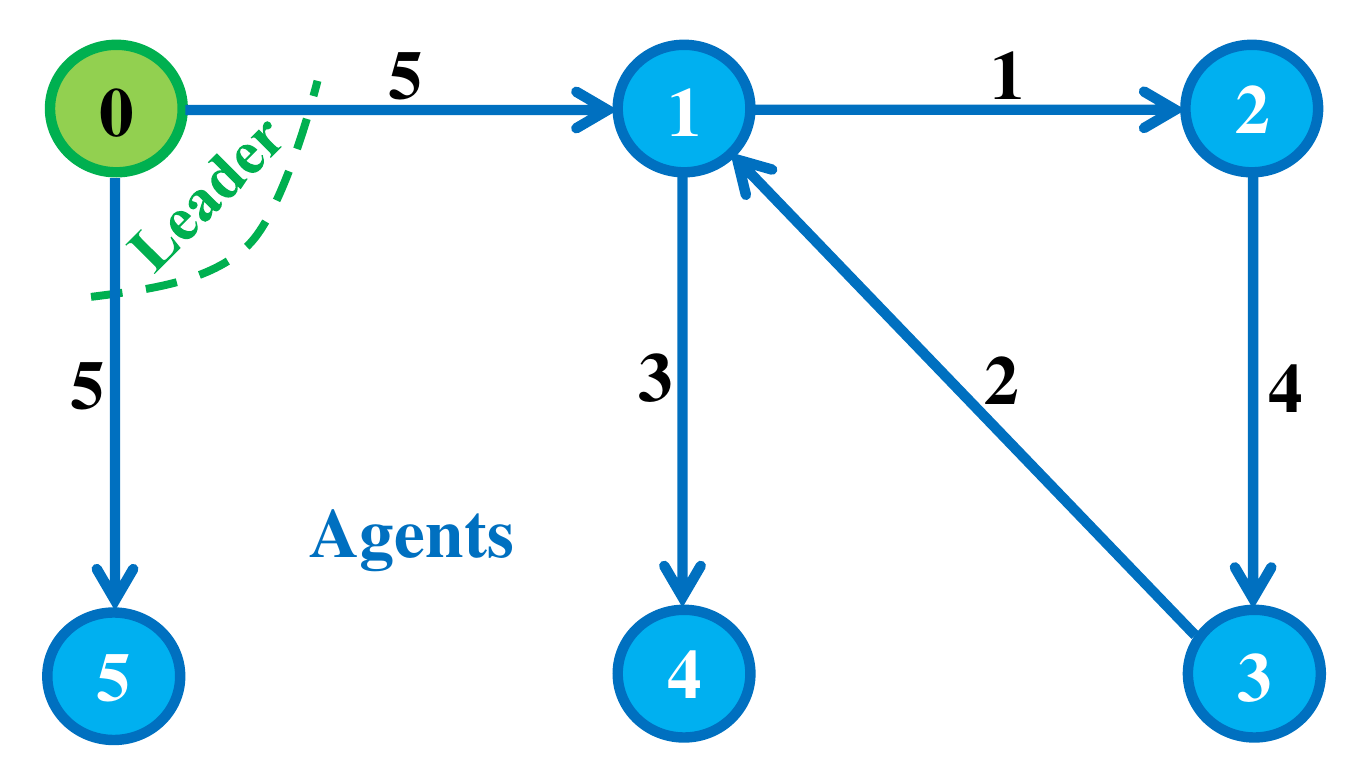} \caption{Connected graph with one leader and five agents.}
		\label{fig:fig2} 
	\end{figure}

	\textbf{Example 1 (SISO):} Let the graph in Fig. \ref{fig:fig2} be
	SISO with high order nonlinear dynamics such as 
	\[
	\begin{aligned}\dot{x}_{i}^{1} & =x_{i}^{2}\\
	\dot{x}_{i}^{2} & =x_{i}^{3}\\
	\dot{x}_{i}^{3} & =f_{i}\left(x_{i}\right)+u_{i}\\
	y_{i} & =x_{i}^{1}
	\end{aligned}
	\]
	such that $i=1,2,\ldots,5$ with nonlinear dynamics 
	\[
	\begin{aligned}f_{1}= & x_{1}^{2}{\rm sin}\left(x_{1}^{1}\right)+{\rm cos}\left(x_{1}^{3}\right)^{2},\\
	f_{2}= & -\left(x_{2}^{1}\right)^{2}x_{2}^{2}+0.01x_{2}^{1}-0.01\left(x_{2}^{1}\right)^{3},\\
	f_{3}= & x_{3}^{2}+{\rm sin}\left(x_{3}^{3}\right),\\
	f_{4}= & -3\left(x_{4}^{1}+x_{4}^{2}-1\right)^{2}\left(x_{4}^{1}+x_{4}^{2}+x_{4}^{3}-1\right)-x_{4}^{3}\\
	& +0.5{\rm sin}\left(2t\right)+{\rm cos}\left(2t\right),\\
	f_{5}= & {\rm cos}\left(x_{5}^{1}\right)
	\end{aligned}
	\]
	and the leader dynamics is 
	\[
	\begin{aligned}\dot{x}_{0}^{1}= & x_{0}^{2}\\
	\dot{x}_{0}^{2}= & x_{0}^{3}\\
	\dot{x}_{0}^{3}= & -x_{0}^{2}-2x_{0}^{3}+1+3{\rm sin}\left(2t\right)+6{\rm cos}\left(2t\right)\\
	& -\frac{1}{3}\left(x_{0}^{1}+x_{0}^{2}-1\right)\left(x_{0}^{1}+4x_{0}^{2}+3x_{0}^{3}-1\right)\\
	y_{0}= & x_{0}^{1}
	\end{aligned}
	\]
	The setting parameters in this problem were selected as $\rho_{\infty}=0.03\times{\bf \underline{1}}_{5\times1}$,
	$\rho_{0}=5\times{\bf \underline{1}}_{5\times1}$, $\ell=0.6\times{\bf \underline{1}}_{5\times1}$,
	$\bar{\delta}=4\times{\bf \underline{1}}_{5\times1}$, $\underline{\delta}=5\times{\bf \underline{1}}_{5\times1}$,
	$c=30\times{\bf \underline{1}}_{5\times1}$, $k=0.1$, $\Gamma_{1}=10000\mathbf{I}_{5\times1}$,
	$\Gamma_{2}=10000\mathbf{I}_{5\times1}$ . $x_{0}\left(0\right)=[0.3,0.3,0.3]^{\top}$
	is the initial vector of the nonlinear leader system, and the values
	of agents are 
	\[
	\begin{aligned}
	x_{1}\left(0\right)&=[-0.2850,-0.0821,-0.2126]^{\top},\\
	x_{2}\left(0\right)&=[-0.6044,-0.3964,-0.0775]^{\top},\\
	x_{3}\left(0\right)&=[-0.2110,-0.4237,-0.3253]^{\top},\\
	x_{4}\left(0\right)&=[-0.1501,-0.3986,-0.0050]^{\top},\\
	x_{5}\left(0\right)&=[-0.3281,0.1618,-0.4160]^{\top}
	\end{aligned}
	\]
	Fig. \ref{fig:figSISO8} presents the output performance of the proposed
	control and it shows smooth tracking performance for $y=x^{1}$. It
	can be noticed that the networked system with unknown high nonlinear
	dynamics converged smoothly to the leader trajectory in the presence
	of high nonlinearities and time-varying disturbances. Fig. \ref{fig:figSISO9}
	illustrates the systematic convergence of the synchronized error $e_{i}$
	for $i=1,\ldots,5$ satisfying the predefined constraints and setting
	parameters imposed on the system. In fact, Fig. \ref{fig:figSISO9}
	shows how the error started from a predefined large set and reduced
	systematically into the predefined small set prescribed by the value
	of $\rho_{\infty}$. Moreover, Fig. \ref{fig:figSISO9} presents transformed
	error $\varepsilon_{i}$ associated to agent $i$. 
	\begin{figure}[h!]
		\includegraphics[scale=0.26]{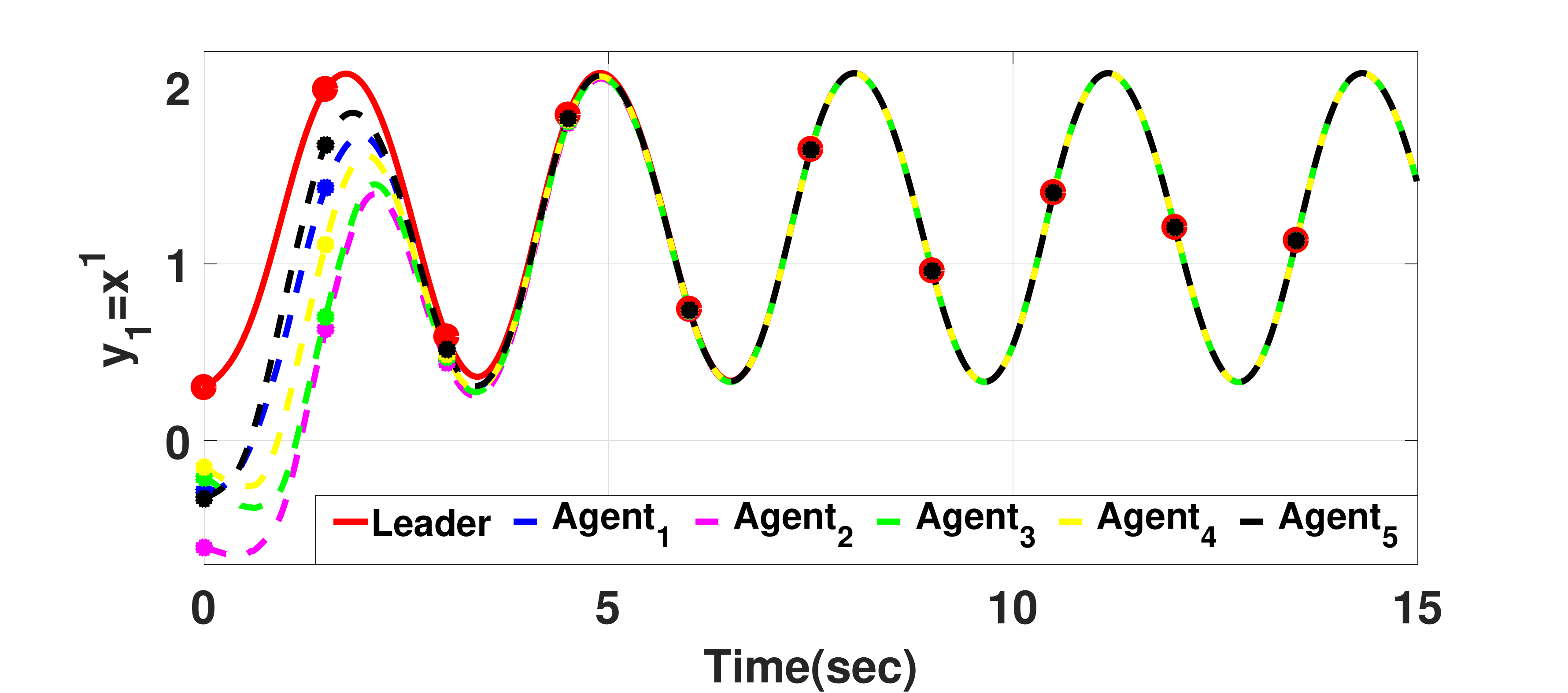} \caption{The output performance of high order nonlinear SISO networked system
			$y=\left[y_{0},y_{1},\ldots,y_{5}\right]$.}
		\label{fig:figSISO8} 
	\end{figure}
	
	\begin{figure}[h!]
		\centering \includegraphics[scale=0.26]{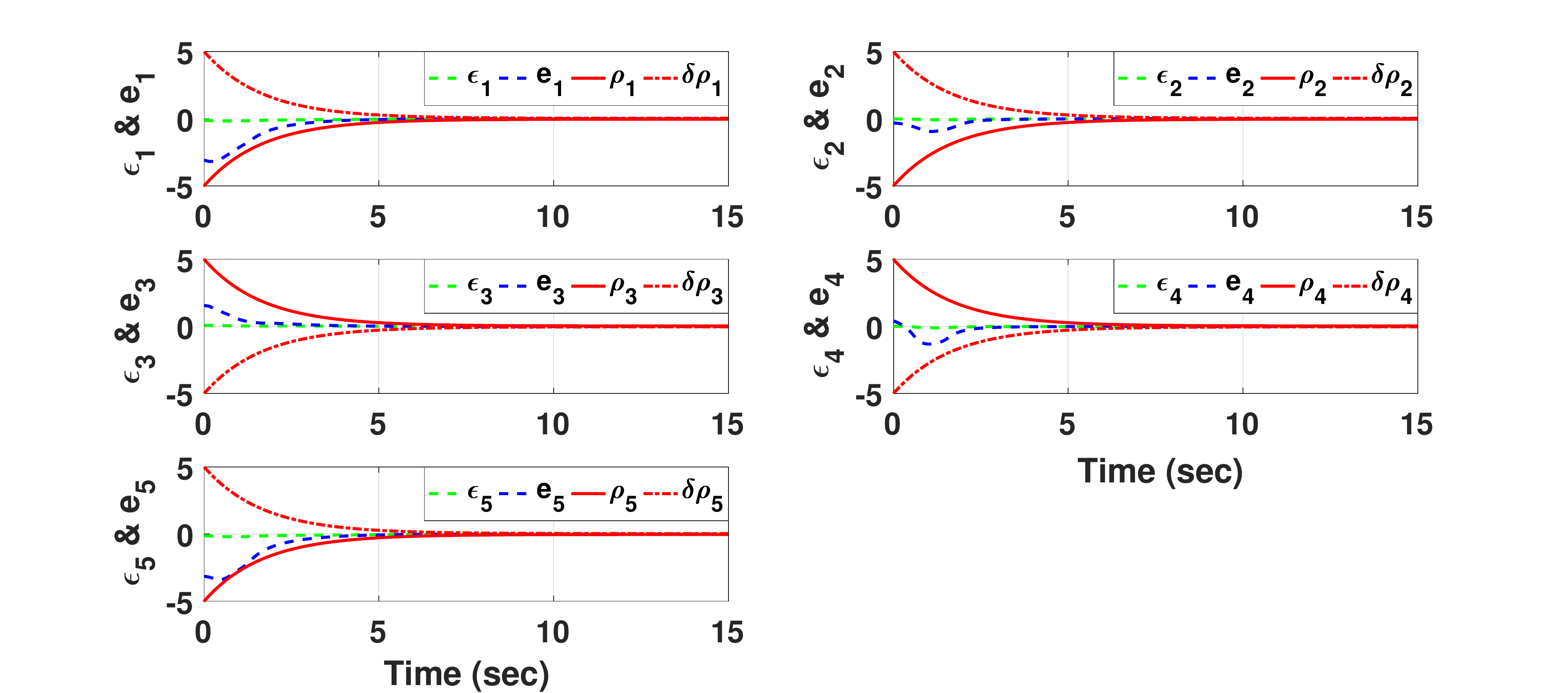} \caption{Error and transformed error of high order nonlinear SISO networked
			system.}
		\label{fig:figSISO9} 
	\end{figure}
	
	\textbf{Example 2 (MIMO Problem):} Consider the graph in Fig. \ref{fig:fig2}
	and let each agent be defined by a second order system with 2 inputs
	and 2 outputs with high nonlinear dynamics. The nonlinear dynamics
	are defined by 
	\[
	\begin{aligned}\begin{bmatrix}\ddot{x}_{i}^{1}\\
	\ddot{x}_{i}^{2}
	\end{bmatrix} & =\begin{bmatrix}f_{i}^{1}\left(x_{i},t\right)\\
	f_{i}^{2}\left(x_{i},t\right)
	\end{bmatrix}+\psi_{i}\left(t\right)\begin{bmatrix}x_{i}^{1}\\
	x_{i}^{2}
	\end{bmatrix}+\begin{bmatrix}D_{i}^{1}\left(t\right)\\
	D_{i}^{2}\left(t\right)
	\end{bmatrix}+\begin{bmatrix}u_{i}^{1}\left(t\right)\\
	u_{i}^{1}\left(t\right)
	\end{bmatrix}\\
	y_{i,:} & =\left[x_{i}^{1},x_{i}^{2}\right]^{\top}
	\end{aligned}
	\]
	where 
	\[
	\begin{aligned}f_{i}\left(x_{i}\right)=\begin{bmatrix}a_{i}^{1}x_{i}^{2}\left(x_{i}^{1}\right)^{2}\dot{x}_{i}^{2}+0.2{\rm sin}\left(a_{i}^{1}x_{i}^{1}\dot{x}_{i}^{1}\right)\\
	-a_{i}^{2}x_{i}^{1}x_{i}^{2}\dot{x}_{i}^{1}-0.2a_{i}^{2}{\rm cos}\left(a_{i}^{2}x_{i}^{2}t\right)x_{i}^{1}\dot{x}_{i}^{2}
	\end{bmatrix},\end{aligned}
	\]
	\[
	\begin{aligned}\psi_{i}\left(t\right)=\begin{bmatrix}3c_{i}^{1}{\rm sin}\left(0.5t\right) & 2c_{i}^{1}{\rm sin}\left(0.4c_{i}^{1}t\right){\rm cos}\left(0.3t\right)\\
	0.9{\rm sin}\left(0.2c_{i}^{2}t\right) & 2.5{\rm sin}\left(0.3c_{i}^{2}t\right)+0.3cos\left(t\right)
	\end{bmatrix},\end{aligned}
	\]
	\[
	\begin{aligned}D_{i}\left(t\right)=\begin{bmatrix}1+b_{i}^{1}{\rm sin}\left(b_{i}^{1}t\right)\\
	1.2{\rm cos}\left(b_{i}^{2}t\right)
	\end{bmatrix},\end{aligned}
	\]
	and 
	\[
	\begin{aligned}a=\begin{bmatrix}a_{i}^{1}\\
	a_{i}^{2}
	\end{bmatrix}=\begin{bmatrix}1.5 & 0.5 & 0.7 & 1.3 & 0.7\\
	0.5 & 1.4 & 0.1 & 1.3 & 2.4
	\end{bmatrix}^{\top},\\
	b=\begin{bmatrix}0.5 & 1.5 & 1.1 & 1.6 & 0.3\\
	0.7 & 1.2 & 1.3 & 0.5 & 0.3
	\end{bmatrix}^{\top},\\
	c=\begin{bmatrix}1.5 & 2.5 & 0.5 & 1.7 & 0.7\\
	0.5 & 1.7 & 1.1 & 0.3 & 0.4
	\end{bmatrix}^{\top}
	\end{aligned}
	\]
	where $f_{i}\left(\cdot\right)\in\mathbb{R}^{2}$ is system nonlinearity,
	$D_{i}\left(t\right)\in\mathbb{R}^{2}$ and $\psi_{i}\left(t\right)\in\mathbb{R}^{2\times2}$
	are time variant disturbances of the associated agent $i$. Also,
	the correspondent nonlinear and time variant disturbance components
	$f_{i}\left(\cdot\right)$, $D_{i}\left(t\right)$ and $\Upsilon_{i}\left(t\right)$
	are assumed to be completely unknown. It can be noticed that we considered
	$f_{i}\left(x_{i}\right)\neq f_{j}\left(x_{j}\right)$, $D_{i}\left(t\right)\neq D_{j}\left(t\right)$,
	$\psi_{i}\left(t\right)\neq\psi_{j}\left(t\right)$, for $i\neq j$
	and $i,j=1,\cdots,N$. The leader dynamics is selected to be $x_{0}=[0.6{\rm cos}(0.6t),0.8{\rm cos}(0.5t)]^{\top}$.
	The control parameters of the problem were defined as $\rho_{\infty}=0.03\times{\bf {\bf \underline{1}}_{5\times2}}$,
	$\rho_{0}=7\times{\bf {\bf \underline{1}}_{5\times2}}$, $l=0.6\times{\bf {\bf \underline{1}}_{5\times2}}$,
	$\Gamma_{1}=100\mathbf{I}_{5\times2}$, $\Gamma_{2}=100\mathbf{I}_{5\times2}$,
	$\bar{\delta}=7\times{\bf {\bf \underline{1}}_{5\times2}}$, $\underline{\delta}=7\times{\bf {\bf \underline{1}}_{5\times2}}$,
	$c=30\mathbf{I}_{5\times2}$, $k=0.01\mathbf{I}_{5\times2}$. Initial
	conditions of $x_{1}\left(0\right)=[0.2310,-0.0276]^{\top}$, $x_{2}\left(0\right)=[-0.1362,-0.4615]^{\top}$,
	$x_{3}\left(0\right)=[-0.2867,-0.1315]^{\top}$, $x_{4}\left(0\right)=[0.5157,0.3539]^{\top}$,
	$x_{5}\left(0\right)=[-0.4700,0.4070]^{\top}$ and the derivative
	of states is $\dot{x}\left(0\right)={\rm randn\left(5,2\right)}$
	where ''randn'' is a command in $\text{MATLAB}^{\circledR}$ denotes
	normal random distribution. It can be noticed that we considered different
	initial conditions such that $x_{i}\left(0\right)\neq x_{j}\left(0\right)$,
	for $i\neq j$ and $i,j=1,\cdots,N$.\\
		\begin{figure}[h!]
		\centering \includegraphics[scale=0.26]{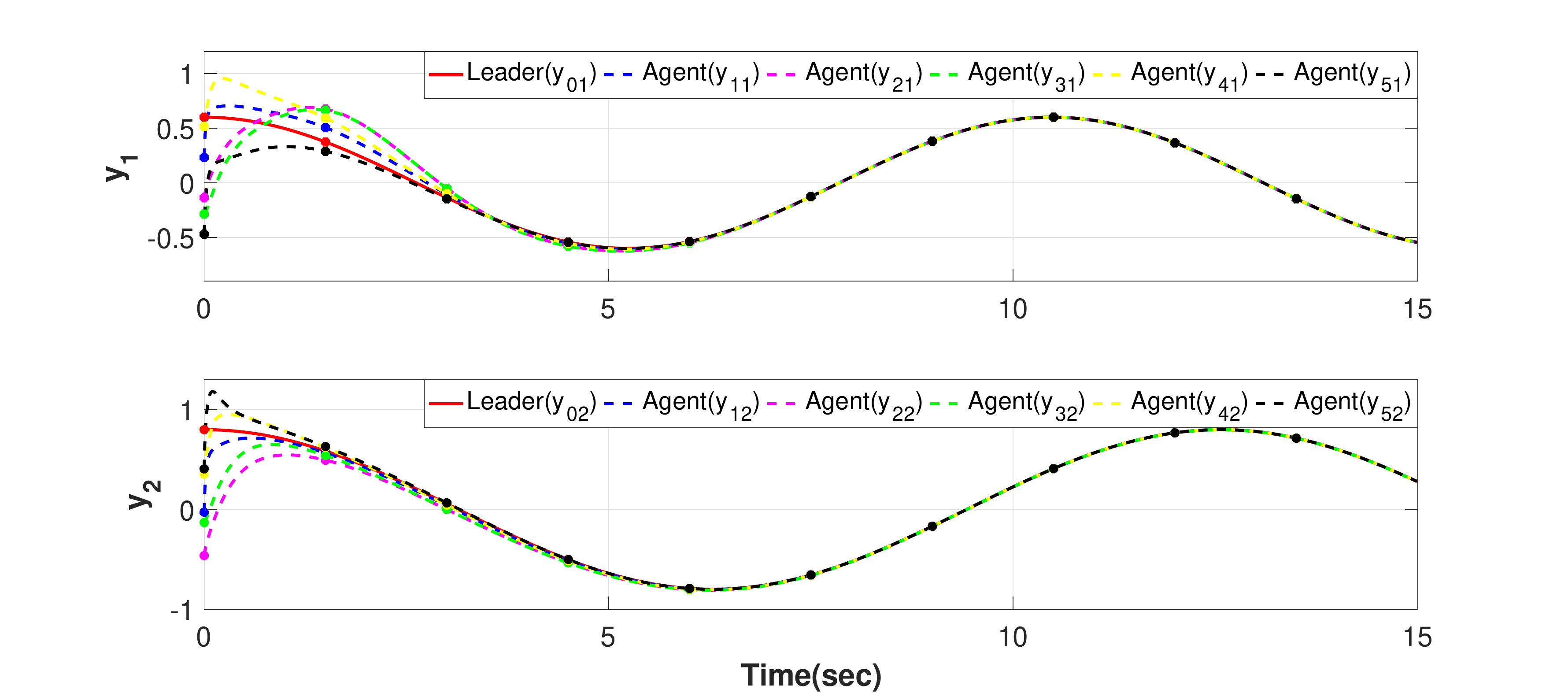} \caption{Output performance of MIMO nonlinear networked system for $y_{1}$
			and $y_{2}$.}
		\label{fig:fig7} 
	\end{figure}
	
	\begin{figure}[h!]
		\centering \includegraphics[scale=0.26]{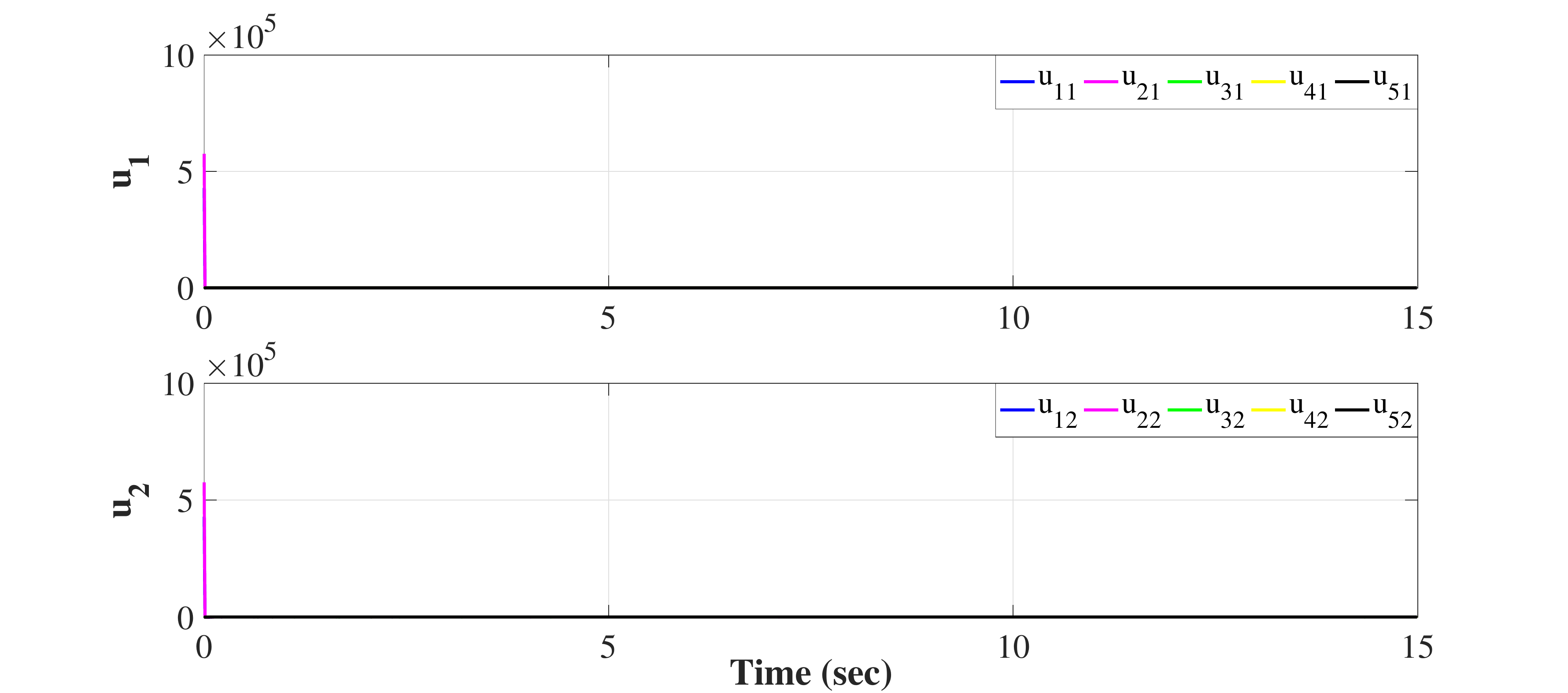} \caption{Control signal of MIMO nonlinear networked system for $u_{1}$ and
			$u_{2}$.}
		\label{fig:fig8} 
	\end{figure}
	
	\begin{figure}[h!]
		\centering \includegraphics[scale=0.26]{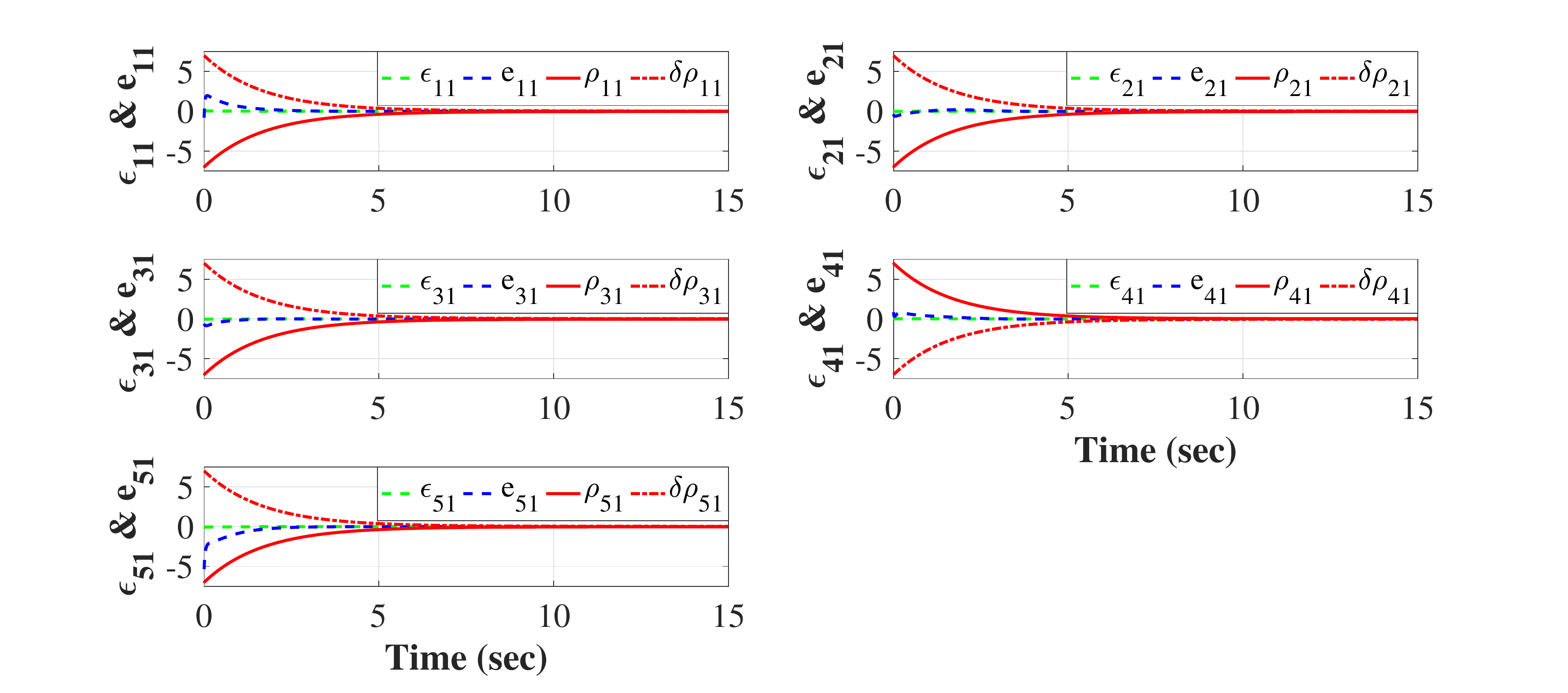} \caption{Error and transformed error of MIMO nonlinear networked system for
			the first output.}
		\label{fig:fig9} 
	\end{figure}
	
	\begin{figure}[h!]
		\centering \includegraphics[scale=0.26]{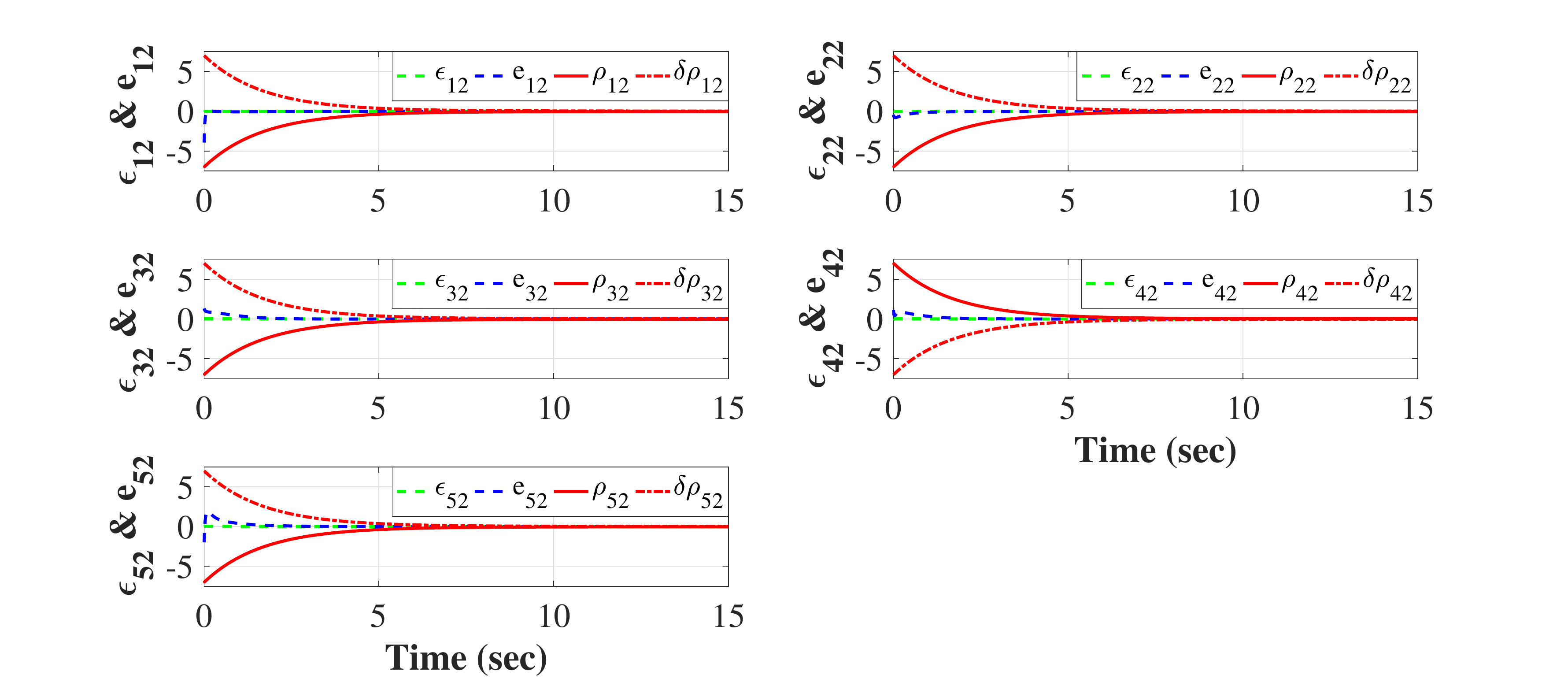} \caption{Error and transformed error of MIMO nonlinear networked system for
			the second output.}
		\label{fig:fig10} 
	\end{figure}

	\begin{figure*}
	\centering \includegraphics[scale=0.45]{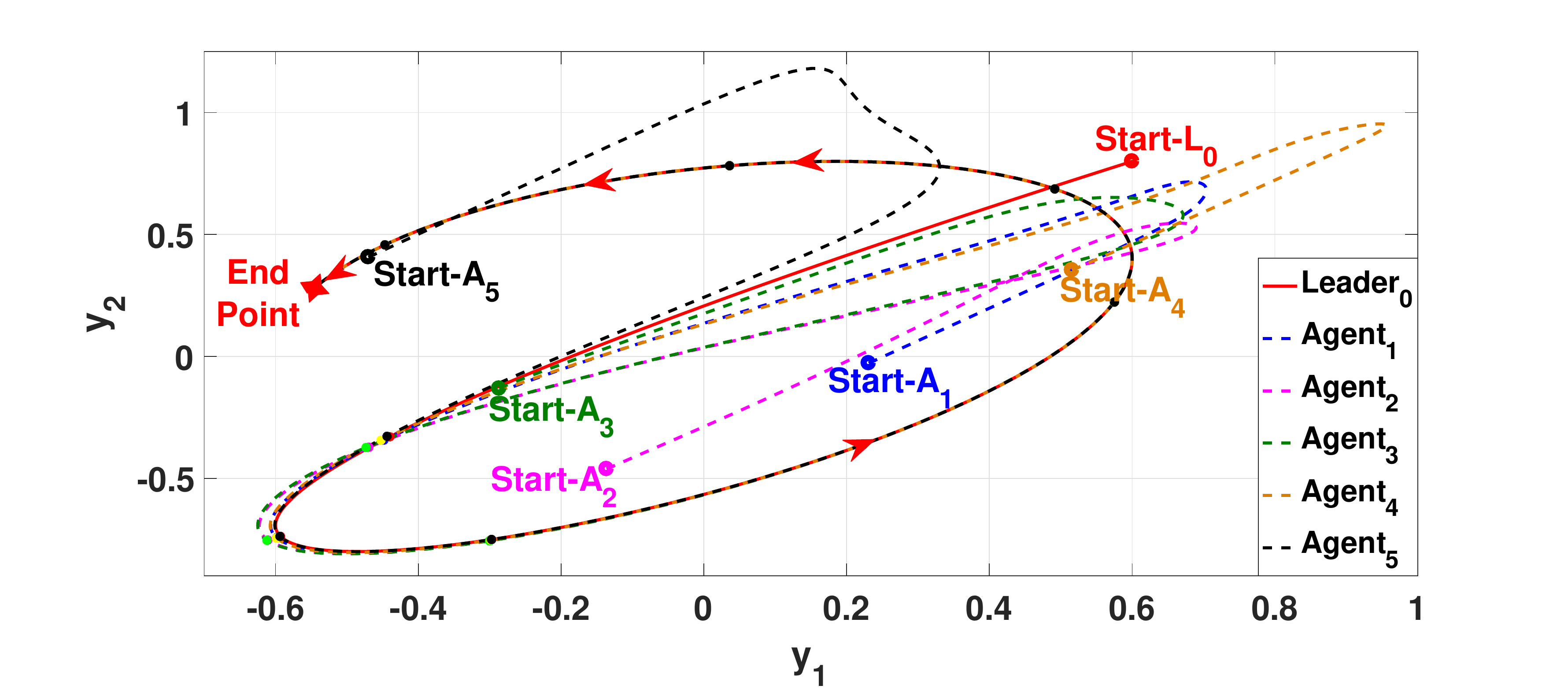} \caption{Phase plane plot for all the agents.}
	\label{fig:fig11} 
    \end{figure*}

	The robustness of the proposed controller against time variant uncertain
	parameters, time variant external disturbances, and high nonlinearities
	are tested. Fig. \ref{fig:fig7} shows the output performance of the
	proposed controller for this MIMO case. The output performance demonstrate
	impressive smooth tracking performance of the follower agents $x_{i}$
	for $i=1,\ldots,N$ to the leader agent $x_{0}$. The control input
	of system dynamics in the connected graph is demonstrated in Fig.
	\ref{fig:fig8}. The tracking errors and their associated transformed
	errors are depicted in Fig. \ref{fig:fig9} and \ref{fig:fig10}.
	The initial errors of $e_{i1}\left(0\right)$ and $e_{i2}$$\left(0\right)$
	satisfy the static boundaries that were defined initially by $\bar{\delta}$and
	$\underline{\delta}$. The transient errors $e_{i1}\left(t\right)$
	and $e_{i2}\left(t\right)$ for $t\geq0$ satisfy the dynamic boundaries
	which can be represented by $\rho_{i1}\left(t\right)$ and $\rho_{i2}\left(t\right)$,
	respectively. Finally, the errors are successfully trapped between
	$\pm\rho_{ij\infty}$ and $\delta\rho_{i,j}$ for $i=1,\ldots,N$
	and $j=1,2$. In fact, Fig. \ref{fig:fig7}, \ref{fig:fig8}, \ref{fig:fig9},
	and \ref{fig:fig10} illustrate the effectiveness and robustness of
	the proposed distributed adaptive control with prescribed performance
	characteristics of the high order unknown nonlinear networked systems.
	The phase plane is presented in Fig. \ref{fig:fig11}.

	\section{Conclusion \label{sec8}}
	
	A distributed cooperative adaptive control of high order nonlinear
	multi-agent systems with prescribed performance has been proposed.
	We considered the nonlinearities of the agents to be Lipschitz but
	completely unknown. Adaptive estimates has been used to estimate the
	system's nonlinearities and uncertainties. The controller has been
	developed for strongly connected digraph. The control signal has been
	chosen properly to ensure UUB stability of the synchronization errors.
	The controller successfully brought the agents from random initial
	values and synchronized their outputs to the desired trajectory within
	the prescribed performance constraints. Simulation examples included
	highly nonlinear heterogeneous systems with time varying parameters,
	uncertainties and disturbances.
	
	\section*{Acknowledgment}
	
	The author would like to thank \textbf{Babajide O Ayinde} for providing
	helpful comments that greatly improved the manuscript and for his
	assistance that made this work possible.
	
	
	
	
	\ifCLASSOPTIONcaptionsoff \newpage\fi
	
	
	
	\bibliographystyle{IEEEtran}
	\bibliography{arXiv_Hashim_PPF}
	
	%
	
	
	
	
	
\end{document}